\documentclass[11pt,reqno]{amsart}
\usepackage{amsmath,amsthm,amsfonts,amssymb,xcolor,comment,hyperref,fontawesome}
\usepackage{graphicx}
\usepackage{setspace}
\usepackage{mathrsfs}
\usepackage{hyperref}
\usepackage[letterpaper,margin=1.3in]{geometry}

\usepackage[shortlabels]{enumitem}

\usepackage{comment}

\usepackage[
    backend=biber,
    style=numeric,
    doi=false,isbn=false,url=false,eprint=false
    ]{biblatex}
\renewcommand{\epsilon}{\varepsilon}
\renewcommand{\hat}{\widehat}
\renewcommand{\tilde}{\widetilde}

\def\XXint#1#2#3{{\setbox0=\hbox{$#1{#2#3}{\int}$}
     \vcenter{\hbox{$#2#3$}}\kern-.5\wd0}}

\newtheorem{theorem}{Theorem}[section]
\newtheorem*{theorem*}{Theorem}
\newtheorem{proposition}[theorem]{Proposition}
\newtheorem{lemma}[theorem]{Lemma}
\newtheorem{corollary}[theorem]{Corollary}

\theoremstyle{definition}
\newtheorem{definition}[theorem]{Definition}
\newtheorem{construction}[theorem]{Construction}

\theoremstyle{remark}
\newtheorem{remark}[theorem]{Remark}

\numberwithin{equation}{section}

\newcommand{\bbG}{\mathbb{G}}
\newcommand{\bbR}{\mathbb{R}}

\title[A Lipschitz curve that is purely $C_{H}^{1}$-unrectifiable]{A Lipschitz curve in a Carnot group that is purely unrectifiable by smooth horizontal curves}

\author[Gareth Speight]{Gareth Speight}
\address{Department of Mathematical Sciences, University of Cincinnati, 2815 Commons Way, Cincinnati, OH 45221, United States}
\email{Gareth.Speight@uc.edu}

\author[Scott Zimmerman]{Scott Zimmerman}
\address{Department of Mathematics, 
The Ohio State University,
100 Math Tower, 231 West 18th Avenue, Columbus, OH 43210, United States}
\email{zimmerman.416@osu.edu}


\allowdisplaybreaks

\begin{document}

\begin{abstract}
We construct a Lipschitz curve in the free Carnot group of step 3 with 2 generators that meets every $C^{1}$ horizontal curve in a set of measure zero. This shows that the $C^{1}_{H}$-Lusin property fails in a strong sense in this group, and we deduce that such a curve must be purely $C^1_H$ 1-unrectifiable. Hence 1-rectifiability in Carnot groups is wildly different to its counterpart in Euclidean spaces, wherein the Whitney Extension Theorem guarantees that Lipschitz rectifiability and $C^1$ rectifiability are equivalent.
\end{abstract}


\date{\today}

\maketitle

\section{Introduction}

A set $E$ in a metric space $X$ is {\em $k$-rectifiable} if it can be covered by a countable union of Lipschitz images of measurable subsets of $\mathbb{R}^k$ up to a set of $\mathcal{H}^k$-measure zero, where $\mathcal{H}^k$ denotes $k$-dimensional Hausdorff measure.
This is a metric analogue of a smooth manifold, where Lipschitz maps play the role of ``non-smooth charts''. On the other end of the rectifiability spectrum, a set $E \subset X$ is {\em purely $k$-unrectifiable} if the intersection of $E$ with {\em any} Lipschitz image of a measurable subset of $\mathbb{R}^k$ has $\mathcal{H}^k$-measure zero. 
A nontrivial example of a purely 1-unrectifiable set in $\mathbb{R}^2$ is the four-corner Cantor set \cite{Gar70}.

For subsets of $\mathbb{R}^n$, an equivalent definition of $k$-rectifiability is obtained by using $C^{1}$ images of open subsets of $\mathbb{R}^{k}$. This equivalence follows from Kirszbraun’s Extension Theorem and the classical Lusin and Whitney Extension Theorems (see \cite[Theorem~3.1.16]{Fed69} or Lemma~\ref{l-Federer} below). Other equivalent definitions of rectifiability exist for sets in $\mathbb{R}^n$ using, for example, approximate tangent planes \cite{EG15}, $\mathcal{H}^k$-density \cite{P87}, projections \cite{B64}, or Jones' square functions \cite{AT15}. Often $k$-rectifiability is easiest to understand when $k=1$. For example, in any complete metric space,
a closed, connected set with finite $\mathcal{H}^1$-measure must be 1-rectifiable \cite[Theorem~4.4.8]{AT04}.

A Carnot group is a Lie group whose Lie algebra is nilpotent and stratified. They arise naturally in the study of local properties of sub-Riemannian manifolds, are important in control theory, and are rich objects of study in their own right \cite{BLU07, Mon02}. Carnot groups possess a great deal of geometric structure; they admit translations, dilations, a Haar measure, and a geodesic distance defined by minimizing lengths of {\em horizontal} (or {\em admissible}) curves.
Horizontal curves are those that lie tangent to the first layer of the stratification
and are particularly important to a Carnot group's geometry. 
While every Carnot group is diffeomorphic to a Euclidean space as a smooth manifold, they are wildly different in their metric and geometric properties.
For example, 
the first Heisenberg group $\mathbb{H}$ is diffeomorphic to $\mathbb{R}^3$, but
every set in $\mathbb{H}$ is purely $k$-unrectifiable for all $k > 1$ \cite{HM15}.
As such, it becomes necessary to consider other definitions of rectifiability in Carnot groups, and this has been extensively studied in recent years 
\cite{AM22,DLMV22,FSS01,Mag06,Pau04}.


In this paper, we examine the relationship between (Lipschitz) 1-rectifiability and $C^1_H$ 1-rectifiability in Carnot groups.
 Note that by \cite[Proposition~4.1]{Pan89}
 every Lipschitz curve in a Carnot group must be horizontal.
 We say that a curve in a Carnot group is $C^1_H$ if it is horizontal and is $C^1$ when viewed as a curve in the diffeomorphically equivalent Euclidean space.
Equivalently, a curve is $C^1_H$ if its Pansu derivative exists and is continuous \cite{Mag13,Zim25b}.
As in the Euclidean case, 1-rectifiability and $C^1_H$ 1-rectifiability are equivalent in any Carnot group in which a $C^1_H$ version of Whitney's Extension Theorem holds for curves \cite[Remark~5.3]{JS17}. Such extension results were first proven in \cite{Zim18} in the Heisenberg groups and then in \cite{JS17} for pliable Carnot groups. Further Whitney-type results for curves with higher regularity have been well studied in recent years \cite{CPS21b,CPS21a,PSZ19,PSZ24,SZ23,SZ25,Zim23}.

Our main result shows that the notions of 1-rectifiability and $C^1_H$ 1-rectifiability can be very different in general.

\begin{theorem}\label{t-supermain}
    There is a Carnot group $\mathbb{F}$ and a Lipschitz curve $\gamma\colon [0,1] \to \mathbb{F}$ such that $\gamma([0,1])$ is purely $C^1_H$ 1-unrectifiable.
\end{theorem}

Here pure $C^1_H$ 1-unrectifiability is defined using intersections with images of $C^{1}_{H}$ curves (Definition~\ref{d-pureunrec}).
The image of a Lipschitz curve is the simplest example of a 1-rectifiable set in a metric space, so Theorem \ref{t-supermain} illustrates that there is no hope of marrying the notions of 1-rectifiability and $C^1_H$ 1-rectifiability in general Carnot groups. The Carnot group $\mathbb{F}$ in Theorem \ref{t-supermain} is the ``Cartan group'' i.e. the free Carnot group of step 3 with 2 generators (Definition~\ref{d-freegroup}).
This is one of the five possible Carnot group structures on $\mathbb{R}^5$ (up to isomorphic equivalence)
\cite[Theorem~10.81]{ABB19}.
The word ``free'' describes the Lie algebra of $\mathbb{F}$ as it is generated freely by the Lie bracket operation on a set of two generators in a way that terminates upon the third iteration. There is a rich history behind the Cartan group $\mathbb{F}$ \cite{BH14}.
In particular, one may understand its algebraic structure via the physical scenario of one sphere rolling along the surface of another \cite{Zel06}, where
every configuration is obtainable by combining only two motions: rolling one sphere along lines of latitude and rolling it along lines of longitude.

To prove Theorem~\ref{t-supermain}, we construct a Lipschitz curve $\gamma\colon [0,1] \to \mathbb{F}$ such that, for any $C^1_H$ curve $\Gamma\colon [0,1] \to \mathbb{F}$, we have 
$
\mathcal{H}^1(\gamma([0,1]) \cap \Gamma([0,1]))=0
$.
Existence of such a Lipschitz curve follows from the following result, which is of independent interest.

\begin{theorem}\label{lusinzero}
There is a Lipschitz curve $\gamma\colon [0,1]\to \mathbb{F}$ such that, for every $C^{1}_H$ curve $\Gamma\colon [0,1]\to \mathbb{F}$, we have
\begin{equation}
    \label{e-supergoal}
m\{t\in [0,1]:\Gamma(t)=\gamma(t)\}=0.
\end{equation}
\end{theorem}

We now briefly describe the motivation behind the development of this theorem.
A Carnot group $\bbG$ has the $C^1_H$-Lusin property if, for any Lipschitz curve $\gamma\colon [0,1] \to \bbG$ and any $\varepsilon>0$, there is a $C^1_H$ curve $\Gamma\colon [0,1] \to \bbG$ so that $m\{t\in [0,1]:\Gamma(t)=\gamma(t)\} > 1-\varepsilon$. 
This property holds in Euclidean space
and in every pliable Carnot group (for example, the Heisenberg group) \cite{JS17}. 
The first author showed in \cite{Spe16}  that the $C^1_H$-Lusin property does not hold in the Engel group $\mathbb{E}$. More precisely, given any $\varepsilon > 0$ there is a Lipschitz curve $\gamma\colon [0,1] \to \mathbb{E}$ such that $m\{t\in [0,1]:\Gamma(t)=\gamma(t)\} \leq \varepsilon$ for every $C^1_H$ curve $\Gamma\colon [0,1] \to \mathbb{E}$. Answering a question posed by Prof. Jonathan Bennett, the authors of the present paper showed that, in the Engel group, one cannot obtain a similar statement with $\varepsilon=0$. Instead, every Lipschitz curve in the Engel group meets some $C^{1}_H$
curve in a set of strictly positive measure \cite{SZ25}. 
Here a crucial fact is that $X_2$ is the only direction in the horizontal distribution of $\mathbb{E}$ that is not pliable, 
so a curve whose velocity is bounded away from $X_{2}$ does in fact admit a $C^{1}_{H}$-Lusin approximation. Theorem~\ref{lusinzero} shows that the setting of the Cartan group $\mathbb{F}$ is quite different from that of the Engel group.

To prove Theorem \ref{lusinzero}, we exploit the fact that, in $\mathbb{F}$, both generating directions $X_{1}$ and $X_{2}$ are non-pliable. In coordinates, this gives constraints on the points that are reachable by $C^{1}_H$ 
curves moving in the directions $\pm X_{1}$ or $\pm X_{2}$ (Lemma \ref{unreachable}). For example, any horizontal curve moving in the positive $X_2$ direction \textit{must} be increasing in its 4th coordinate. We use this to construct Lipschitz horizontal curves that become increasingly difficult for $C^{1}_H$ 
curves to follow (Construction \ref{c-construction} and Proposition \ref{badintersect}). Since suitable pairs of points can be joined by a concatenation of eight horizontal segments in directions $\pm X_{1}$ and $\pm X_{2}$ (Proposition~\ref{p-staircase}), this construction can be iterated (Construction \ref{c-inductivedefinition}). We then obtain a limit curve (Proposition \ref{limitislipschitz}) and deduce Theorem \ref{lusinzero}.

We now describe the organization of the paper. In Section \ref{s-carnot} we define the free Carnot group $\mathbb{F}$, describe its structure as a metric space, and define pure $C^1_H$ 1-unrectifiability.
In Section~\ref{reachability}, we show how to move vertically in $\mathbb{F}$ using ``staircase'' curves that are piecewise segments in the horizontal basis directions $\pm X_{1}, \pm X_{2}$. We also establish the relationship between the directions of the horizontal components of a horizontal curve and its changes in height. Section~\ref{nolusin} contains the construction for and proof of Theorem~\ref{lusinzero}.
Finally, in Section~\ref{s-equivalence} we use Theorem~\ref{lusinzero} to prove Theorem~\ref{t-supermain}.

\medskip

\textbf{Acknowledgements:} G. Speight was supported by the National Science Foundation under Award No. 2348715.

\section{Preliminaries}
\label{s-carnot}

In this section we describe our main objects of study. For simplicity we restrict our attention to the free Carnot group, though many of our definitions also make sense in more general Carnot groups. The relevant definitions in the more general setting can be found in \cite{SZ25} or \cite{BLU07}.

\subsection{Lie groups}
A {\em Lie group} is a (smooth, finite dimensional) manifold together with a group operation that is smooth and has smooth inverse. 

Given a Lie group operation on $\mathbb{R}^n$,
a vector field $X$ 
on $\mathbb{R}^n$ is {\em left invariant} if
$
dL_x(y) X(y) = X(L_x(y))
$
where $L_x\colon \mathbb{R}^n \to \mathbb{R}^n$ is the operation of left multiplication by $x$ and $dL_x$ is its differential i.e. its Jacobian matrix.

Given two vector fields $X$ and $Y$ on $\mathbb{R}^n$, the {\em Lie bracket of $X$ and $Y$} is defined to be the vector field $[X,Y] = XY-YX$.

\subsection{The free Carnot group of step 3 with 2 generators} 
There is a Lie group operation $*$ on $\mathbb{R}^5$ for which the following is a basis of left invariant vector fields \cite{BL21}:
\begin{equation}
    \label{e-F23coords}
    \begin{aligned}
        &X_{1}=\partial_{1}, \qquad X_{2}=\partial_{2}-x_{1}\partial_{3}+\tfrac12 x_{1}^{2} \partial_{4}+x_{1}x_{2}\partial_{5},\\
        &X_3 = \partial_3 -x_1 \partial_4 - x_2 \partial_5,
        \qquad
        X_4 = \partial_4,
        \qquad
        X_5 = \partial_5.
    \end{aligned}
\end{equation}
The only non-zero Lie bracket relationships in \eqref{e-F23coords} are
\begin{equation}
    \label{e-stratification}
    [X_{2},X_{1}]=X_{3}, \qquad [X_{3}, X_{1}]=X_{4}, \quad \text{ and } \quad [X_{3}, X_{2}]=X_{5}.
\end{equation}
Therefore, the vector space $\mathfrak{f} = \text{span}\{X_1,\dots,X_5\}$ is closed under Lie bracketing and is hence an example of a {\em Lie algebra}. Since the Lie brackets of the vectors in \eqref{e-F23coords} vanish after 3 iterations by \eqref{e-stratification}, we also say that $\mathfrak{f}$ is {\em nilpotent} and of step 3.

While we will not need the full formulation of the group operation $*$, we note the following from \cite{BL21}: 
given $x,y,z\in\mathbb{R}^5$ with $z = x*y$,
\begin{equation}
\label{e-groupoperation}
z_1 = x_1 + y_1 \quad \text{and} \quad
z_2 = x_2 + y_2.
\end{equation}
The other coordinates $z_i$ of $z$ are polynomials of $x_j$ and $y_k$ with $k,j \leq i$.

\begin{definition}
\label{d-freegroup}
We say that $(\mathbb{R}^5,*)$ is the \emph{free Carnot group of step 3 with 2 generators} and we denote it by $\mathbb{F}$. 
 The Lie algebra $\mathfrak{f}$ is the {\em free-nilpotent Lie algebra of step 3 with 2 generators}. 
\end{definition}

Note that we can write $\mathfrak{f} = V_1 \oplus V_2 \oplus V_3$ where
    \begin{equation}
        \label{e-stratification2}
        V_{1}=\mathrm{Span}\{X_{1}, X_{2}\}, \quad V_{2}=\mathrm{Span}\{X_{3}\}, \quad V_{3}=\mathrm{Span}\{X_{4}, X_{5}\},
    \end{equation}
with $[V_1,V_1] = V_2$ and $[V_2,V_1] = V_3$.
Such a decomposition is called a {\em stratification} of the Lie algebra,
so $\mathbb{F}$ is indeed an example of a step 3 Carnot group. 
For an exhaustive overview of Carnot groups, see Parts I and III of \cite{BLU07}.

The {\em exponential map} $\exp \colon \mathfrak{f} \to \mathbb{R}^5$ is defined as $\exp(X) = \gamma_X(1)$, where $\gamma_X\colon \mathbb{R} \to \mathbb{R}^5$ is the solution to the differential equation $\gamma'(t) = X(\gamma(t))$ with $\gamma(0)=0$. Since $\mathfrak{f}$ is nilpotent, the exponential map is a diffeomorphism.
From \eqref{e-F23coords},
we can then calculate
    \begin{equation}
        \label{e-expformula}
        \exp(tX_1) = (t,0,0,0,0) \quad \text{ and } \quad \exp(tX_2) = (0,t,0,0,0)
        \qquad
        \text{for all } t \in \mathbb{R}
    \end{equation}
    since $\tfrac{d}{dt} (t,0,0,0,0) = X_1(t,0,0,0,0)$
    and
    $\tfrac{d}{dt} (0,t,0,0,0) = X_2(0,t,0,0,0)$.

We could equivalently define $\mathbb{F}$ to be the Carnot group with Lie bracket relations 
satisfying 
\eqref{e-stratification2}. To do this, first define $\mathfrak{g}$ to be any 5 dimensional Lie algebra 
with a basis $\{ Y_1,\dots,Y_5\}$ satisfying the same bracket relationships as \eqref{e-stratification}.
Associated with this nilpotent Lie algebra is a connected, simply connected Lie group $(\mathbb{G},\cdot)$ such that $\mathfrak{g}$ is its space of left invariant vector fields.
Identifying the basis vectors $\{Y_1,\dots,Y_5\} \subset \mathfrak{g}$ with $\{X_1,\dots,X_5\} \subset \mathfrak{f}$ defines a Lie algebra isomorphism between $\mathfrak{g}$ and $\mathfrak{f}$, and this induces a Lie group isomorphism between $\mathbb{G}$ and $\mathbb{F}$.

We point out that $\mathbb{F}$ is the representation of the free Carnot group of step 3 with 2 generators in {\em exponential coordinates of the second kind} \cite{BL21}.
This is because each point $(x_1,\dots,x_5) \in \mathbb{F}$ is equal to
\begin{equation}
\label{e-secondexp}    
\exp(x_{5}X_{5}) * \exp(x_{4}X_{4}) * \exp(x_{3}X_{3}) * \exp(x_{2}X_{2}) * \exp(x_{1}X_{1}).
\end{equation}


\subsection{Horizontal curves in the free Carnot group}

For any absolutely continuous curve $\gamma \colon [a,b]\to \mathbb{R}^5$, 
the velocity $\gamma'(t)$ exists for almost every $t\in [a,b]$.
In keeping with the terminology of sub-Riemannian geometry, we have the following:
\begin{definition}
\label{d-horizontal}
An absolutely continuous curve $\gamma\colon [a,b]\to \mathbb{R}^5$ is 
\emph{horizontal} 
in $\mathbb{F}$ if there exist $u_1,u_2 \in L^1([a,b])$ such that
\begin{equation}\label{dotgamma}
\gamma'(t) = u_1(t)X_{1}(\gamma(t)) + u_2(t)X_{2}(\gamma(t))
\end{equation}
for almost every $t \in [a,b]$. 
 We define the {\em horizontal speed} of $\gamma$ as $|\gamma'|_H := \sqrt{u_1^2+u_2^2}$ whenever \eqref{dotgamma} holds, and 
 note $|\gamma'|_H \in L^1([a,b])$. Define the
 \emph{length} of $\gamma$ as $\ell(\gamma):=\int_{a}^{b} |\gamma'|_H$.
\end{definition}

In other words, a curve is horizontal if its velocity vectors almost always lie within the first (or {\em horizontal}) layer $V_1$ of the stratification \eqref{e-stratification2}.
By using both \eqref{e-F23coords} and \eqref{dotgamma}, we obtain the following alternate characterization.

\begin{lemma}\label{lifting}
Let $\gamma=(\gamma_1,\gamma_2,\gamma_3,\gamma_4,\gamma_5)\colon [a,b]\to \mathbb{R}^5$ be an absolutely continuous curve. Then $\gamma$ is horizontal in $\mathbb{F}$ if and only if the following equalities all hold almost everywhere on $[a,b]$:
\begin{equation}\label{lifteq}
\gamma_{3}'=-\gamma_{2}'\gamma_{1}, \qquad \gamma_{4}'= \tfrac12 \gamma_{2}'\gamma_{1}^{2}, \qquad \gamma_{5}'=\gamma_{2}'\gamma_{1}\gamma_{2}.
\end{equation}
Moreover, $\ell(\gamma) =\int_{a}^{b} |\gamma'|_H =\int_a^b \sqrt{(\gamma_1')^2 + (\gamma_2')^2}$.
\end{lemma}

\begin{proof}
Suppose \eqref{dotgamma} holds. Using \eqref{e-F23coords} gives for almost every $t\in [a,b]$
\[\gamma'(t) = u_1 (t)\cdot (1,0,0,0,0)+u_2 (t) \cdot (0,1,-\gamma_{1}(t),\tfrac12 \gamma_{1}^{2}(t),\gamma_{1}(t)\gamma_{2}(t)).\]
Examining the first two components implies $u_1(t)=\gamma_{1}'(t)$ and $u_2(t)=\gamma_{2}'(t)$, which gives the claimed formula for $\ell(\gamma)$.
Also, substituting these terms back into the equation above gives \eqref{lifteq}. Conversely, if \eqref{lifteq} holds then \eqref{dotgamma} holds with controls $u_1(t)=\gamma_{1}'(t)$ and $u_2(t)=\gamma_{2}'(t)$ for almost every $t\in [a,b]$.
\end{proof}

\subsection{The free Carnot group as a metric space}

\begin{definition}
The \emph{Carnot-Carath\'{e}odory (CC) metric} between any two points $x, y\in \mathbb{F}$ is defined by
\[d_{c}(x,y):=\inf \{\ell(\gamma)\colon \gamma \mbox{ is a horizontal curve joining } x \mbox{ and }y\}.\]
\end{definition}
By the Chow-Rashevsky Theorem \cite[Theorem~19.1.3]{BLU07}, any two points in $\mathbb{F}$ can be connected by a horizontal curve of finite length. Hence $d$ is indeed a metric.
Since the group operation $*$ acts linearly in the first two coordinates (see \eqref{e-groupoperation}), it follows that $d_{c}$ is
left-invariant, 
i.e. $d_{c}(r*p,r*q) = d_{c}(p,q)$ 
for any $p,q,r \in \mathbb{F}$. 

\begin{remark}
    \label{rem-distance}
    For $V = \pm X_1$ or $V = \pm X_2$, we have $d_c(\exp(tV),\exp(sV)) = |t-s|$ for all $s,t \in \mathbb{R}$. For instance, suppose $V = X_1$. Then $\exp(zV) = (z,0,0,0,0)$ for any $z \in \mathbb{R}$ by \eqref{e-expformula}. Hence the claim follows from the fact that the $\tau \mapsto ((1-\tau)t+\tau s,0,0,0,0)$ for $\tau \in [0,1]$ 
    is a horizontal curve from $\exp(tV)$ to $\exp(sV)$, that the length of any horizontal curve is equal to the length of its projection to the first two coordinates (Lemma~\ref{lifting}), and that the shortest curve between two points in the Euclidean plane is a segment.
\end{remark}

The CC-metric on $\mathbb{F}$ is not bi-Lipschitz equivalent to the Euclidean metric $|\cdot|$ 
on $\mathbb{R}^5$. However, 
for any compact set $K\subset \mathbb{F}$ there are constants $c_1,c_2>0$ such that
\begin{equation}\label{disugdist}
c_1|x-y|\leq d_{c}(x,y)\leq c_2|x-y|^{\frac{1}{3}}\qquad \mbox{for all } x,y\in K.
\end{equation}
See, for example, \cite[Proposition~5.15.1]{BLU07}, \cite[Proposition~1.5]{FS82}, or \cite[Theorem~1.5.1]{Mon01}.

 \begin{definition}
 Given an interval $I \subset \bbR$ and a constant $L>0$, a map $\gamma\colon I \to \mathbb{F}$ is an {\em $L$-Lipschitz curve} if it is $L$-Lipschitz continuous with respect to the metric $d_c$ i.e. 
 \[
 d_c(\gamma(a),\gamma(b)) \leq L |a-b|
 \quad
 \text{ for all } a,b \in I.
 \]
 \end{definition}

\begin{remark}
\label{rem-horiz}    
It follows from 
\eqref{disugdist}
that any Lipschitz curve  $\gamma:I \to \mathbb{F}$ is Lipschitz continuous with respect to the Euclidean metric on  any compact interval in $I$, and hence $\gamma'$ exists almost everywhere  in $I$.
Moreover, by \cite[Proposition~4.1]{Pan89} or an argument similar to Proposition~1.1 in \cite{HZ15}, every such curve is necessarily horizontal.
\end{remark}

\begin{definition}
    Given an interval $I \subset \bbR$, a map $\gamma\colon I \to \mathbb{F}$ is a {\em $C^1_H$ curve} if it is horizontal and is $C^1$ as a curve in $\bbR^5$ (i.e. $\gamma'$ exists at every point and is continuous).
\end{definition}
This definition of $C^1_H$ is equivalent to the condition that the Pansu derivative of $\gamma$ exists and is continuous.
See, for example, \cite{Mag13,Zim25b}.
According to the following, these curves are Lipschitz on compact subintervals.

\begin{lemma}
    \label{l-C1HisLip}
 Suppose $I \subset \bbR$ is an interval and $\Gamma\colon I \to \mathbb{F}$ is $C^1_H$.
Let $J\subset I$ be any subinterval on which $\Gamma'$ is bounded. Then $\Gamma|_{J}$ is Lipschitz.
In particular, $\Gamma$ is Lipschitz on any compact subinterval of $I$.

\end{lemma}
\begin{proof}
Set $L:= \sup_{t\in J} |\Gamma'(t)|_H$. 
Then for all $s \leq t$ in $J$, since  $\Gamma|_{[s,t]}$  is a horizontal curve from $\Gamma(s)$ to $\Gamma(t)$, we have
\[
d_c(\Gamma(s),\Gamma(t)) \leq \ell(\Gamma|_{[s,t]})
= \int_s^t |\Gamma'|_H 
\leq L (t-s).
\]
\end{proof}

\subsection{Rectifiability in the free Carnot group}

We now explicitly define pure unrectifiability by $C^1_H$ curves as used in Theorem \ref{t-supermain}.

\begin{definition}
\label{d-pureunrec}
A set $E \subset \mathbb{F}$ is {\em purely $C^1_H$ 1-unrectifiable} if 
$$
\mathcal{H}^1(E \cap \Gamma([0,1])) = 0
$$
for every $C^1_H$ curve $\Gamma\colon [0,1] \to \mathbb{F}$.
\end{definition}

Note that this definition does not depend on the choice of interval in the domain of $\Gamma$, since horizontal curves remain horizontal upon rescalings and translations of their domains. We 
emphasize that the Hausdorff measure here is defined with respect to the CC-metric $d_c$. However, 
if a set has $1$-dimensional Hausdorff measure zero with respect to the CC metric, then it also has $1$-dimensional Hausdorff measure zero with respect to the Euclidean metric. This follows, for instance, from \eqref{disugdist}.

\section{Horizontal Curves in the Cartan group}\label{reachability}

In this section we establish several properties of horizontal curves in $\mathbb{F}$ that will be useful for the construction in Section \ref{nolusin}.

\subsection{Staircase curves}
We begin by explaining how one can travel to points in the third layer of the group using well-behaved curves.
These curves will be piecewise horizontal segments in directions $\pm X_{1}$ or $\pm X_{2}$. 
\begin{definition}
    \label{d-segment}
    Given a vector $V \in \{\pm X_1, \pm X_2\}$ and $\lambda \geq 0$, we say that a curve $\eta\colon [a,b] \to \mathbb{F}$ is a {\em horizontal $\lambda V$-segment} if it has the form
    $$
    \eta(t) 
    = 
    \eta(a) * \exp(\lambda (t-a) V).
    $$
\end{definition}
If the choice of $V$ is not important or not explicitly known, we refer to such a curve as a {\em horizontal $\lambda$-segment}. We emphasize that, by definition, these curves only travel in the directions of the basis vectors $\pm X_{1}, \pm X_{2}$.
For such a curve $\eta$, 
\eqref{e-groupoperation}
and
\eqref{e-expformula}
 give
\begin{equation}
\label{e-length}
\ell(\eta) = \int_a^b \sqrt{(\eta_1')^2+(\eta_2')^2}
=
\int_a^b \lambda
=
\lambda (b-a).
\end{equation}
It also follows from the definition of the exponential map and left invariance of $V$ that $\eta$ is horizontal with 
\begin{equation}\label{e-nicederivativeform}
\eta'(t)=\lambda V(\eta(t)) \mbox{ for all }t\in [a,b].
\end{equation}


\begin{proposition}
\label{p-staircase}
Fix $\lambda\in \mathbb{R}$, and suppose $p \in \mathbb{F}$ is either the point $(0,0,0,\lambda^{3},0)$ or the point $(0,0,0,0,\lambda^{3})$. 
Then there is a curve $\zeta_\lambda\colon [0,8] \to \mathbb{F}$ from $0$ to $p$ such that, for each $k=0,1,\dots,7$,
the restriction $\zeta_\lambda|_{[k,k+1]}$ is a horizontal $|\lambda|$-segment.
In particular, $\ell(\zeta_\lambda) = 8|\lambda|$.
\end{proposition}

In other words, any point of the form $(0,0,0,\lambda^{3},0)$ or $(0,0,0,0,\lambda^{3})$ can be reached from the origin by traversing a ``staircase'' consisting of 8 horizontal segments in the horizontal directions $\pm X_1$ and $\pm X_2$ with a constant speed $|\lambda|$.

To simplify the proof of this result, we will appeal to the 
Baker-Campbell-Hausdorff (BCH) formula, which holds in any Lie group. We state it here in the special case of a step 3 Carnot group (of which $\mathbb{F}$ is an example).
For a reference and proof, see \cite{BLU07}.
\begin{theorem}[Step 3 BCH formula]
\label{t-BCH}
    Given $X,Y \in \mathfrak{f}$, we have
    \begin{align*}
        \exp(X)*\exp(Y) = \exp \left( X + Y +\tfrac12[X,Y] + \tfrac{1}{12} \left([X,[X,Y]] + [Y,[Y,X]]\right) \right).
    \end{align*}
\end{theorem}

\begin{proof}[Proof of Proposition~\ref{p-staircase}]
It suffices to show that $(0,0,0,\lambda^3,0)$ and $(0,0,0,0,\lambda^3)$ can each be written as a product of  $\exp(\lambda V_k)$ for $k = 0,1,\dots,7$, where $V_k \in \{\pm X_1, \pm X_2\}$ for each $k$.
Indeed, after doing so, we can then define the curve $\zeta_\lambda$ so that $\zeta_\lambda(0)=0$ and
\[
\zeta_\lambda(t) = \zeta_\lambda(k) * \exp(\lambda (t-k) V_{k})
\]
for all $k \in \{0,1,\dots,7\}$ and $t \in [k,k+1]$. 

To achieve this, first fix $\lambda, \mu \in \mathbb{R}$. Using the BCH formula from Theorem~\ref{t-BCH} and \eqref{e-stratification},
\[\exp(\mu X_{2})*\exp(\lambda X_{1})=\exp \left( \mu X_{2}+\lambda X_{1}+\tfrac{1}{2}\lambda \mu X_{3} -\tfrac{1}{12}\lambda \mu^{2}X_{5}+ \tfrac{1}{12}\lambda^{2}\mu X_{4} \right)
\]
Replacing $\mu$ by $-\mu$ and $\lambda$ by $-\lambda$, we have
\[
\exp(-\mu X_{2})*\exp(-\lambda X_{1})=\exp \left( -\mu X_{2}-\lambda X_{1}+\tfrac{1}{2}\lambda \mu X_{3} +\tfrac{1}{12}\lambda \mu^{2}X_{5}- \tfrac{1}{12}\lambda^{2}\mu X_{4} \right).
\]
Again applying the BCH formula gives
\[ 
\exp(\mu X_{2})*\exp(\lambda X_{1})*\exp(-\mu X_{2})*\exp(-\lambda X_{1})
=
\exp\left( \lambda \mu X_{3}-\tfrac{1}{2}\lambda^{2}\mu X_{4} - \tfrac{1}{2}\lambda \mu^{2}X_{5} \right).  
\]
Write $F(\lambda,\mu) = \exp(\mu X_{2})*\exp(\lambda X_{1})*\exp(-\mu X_{2})*\exp(-\lambda X_{1})$.
Then we have
\[
F(\lambda, -\lambda)*F(-\lambda, -\lambda)= \exp(\lambda^{3} X_{4})
\]
which implies that $\exp(\lambda^{3} X_{4}) = (0,0,0,\lambda^{3},0)$
can be written as a product of 8 points of the form $\exp(\lambda V)$ for $V \in \{\pm X_1,\pm X_2\}$. 
On the other hand,
\[
F(-\lambda, \lambda)*F(-\lambda, -\lambda) = \exp(\lambda^{3}X_{5})
\]
implying that $\exp(\lambda^{3}X_{5})= (0,0,0,0,\lambda^{3})$ can be written similarly. 
\end{proof}

Since distances in $\mathbb{F}$ are obtained by infimizing over the lengths of horizontal paths, we have the following.
\begin{corollary}
   \label{c-verticalpath}
For any $\lambda \in \mathbb{R}$ and $p \in \mathbb{F}$ of the form $(0,0,0,\lambda^{3},0)$ or $(0,0,0,0,\lambda^{3})$, 
$d_{c}(0,p) \leq 8|\lambda|$.
\end{corollary}

\subsection{Change in heights of a horizontal curve} We next observe how derivatives
of the first two coordinates of a horizontal curve in $\mathbb{F}$ constrain the changes in its other coordinates.
This will be useful when considering $C^{1}_H$ curves moving in directions close to $\pm X_{1}$ or $\pm X_{2}$.

\begin{lemma}\label{unreachable}
Let $\gamma\colon [a,b]\to \mathbb{F}$ be a horizontal curve and fix $s,t \in [a,b]$ with $s<t$. Then the following statements hold. 
\begin{enumerate}
    \item If $\gamma_{1}'\geq 0$ almost everywhere on $[s,t]$, then \[\gamma_{5}(t)-\gamma_{5}(s) \leq \tfrac{1}{2} \gamma_{1}(t)(\gamma_{2}(t))^{2}- \tfrac{1}{2}\gamma_{1}(s)(\gamma_{2}(s))^{2}.\]
    \item If $\gamma_{1}'\leq 0$ almost everywhere on $[s,t]$, then \[\gamma_{5}(t)-\gamma_{5}(s) \geq \tfrac{1}{2} \gamma_{1}(t)(\gamma_{2}(t))^{2}- \tfrac{1}{2}\gamma_{1}(s)(\gamma_{2}(s))^{2}.\]
    \item If $\gamma_{2}'\geq 0$ almost everywhere on $[s,t]$, then $\gamma_{4}(t)-\gamma_{4}(s)\geq 0$.
    \item If $\gamma_{2}'\leq 0$ almost everywhere on $[s,t]$, then $\gamma_{4}(t)-\gamma_{4}(s)\leq 0$.
\end{enumerate}
\end{lemma}

\begin{proof}
We first prove (1) and (2). Applying \eqref{lifteq} and integration by parts yields
\begin{align*}
\gamma_{5}(t)-\gamma_{5}(s)&= \frac{1}{2}\int_{s}^{t}\gamma_{1}(\theta)\frac{d}{d\theta}\left((\gamma_{2}(\theta))^{2}\right) \, d\theta\\
&=\frac{1}{2} [\gamma_{1}(\theta)(\gamma_{2}(\theta))^{2}]_{s}^{t}-\frac{1}{2}\int_{s}^{t}\gamma_{1}'(\gamma_{2}^{2}) \\
&=\frac{1}{2}\gamma_{1}(t)(\gamma_{2}(t))^{2}-\frac{1}{2}\gamma_{1}(s)(\gamma_{2}(s))^{2}-\frac{1}{2}\int_{s}^{t}\gamma_{1}'(\gamma_{2}^{2}).
\end{align*}
From this, (1) and (2) follow immediately.

To see (3) and (4), simply notice that \eqref{lifteq} implies $\gamma_{4}(t)-\gamma_{4}(s)=\tfrac12 \int_{s}^{t}\gamma_{2}'(\gamma_{1}^{2})$. The result then follows.
\end{proof}

\section{Strong Failure of the Lusin Property in the Cartan group}\label{nolusin}

In this section we construct a Lipschitz curve $\gamma\colon [0,1] \to \mathbb{F}$ that meets any $C^{1}_H$ curve $\Gamma\colon [0,1] \to \mathbb{F}$ in a set of zero measure, thereby proving Theorem~\ref{lusinzero}.
The curve $\gamma$ will be the uniform limit of a sequence of Lipschitz curves $\{\gamma_n\}_{n=1}^\infty$ that will be defined inductively. Each curve in the sequence will be a concatenation of an increasing number of horizontal segments providing 
perturbations in the $x_4$ and $x_5$ directions.
Each successive curve will be constructed by replacing every horizontal segment in the previous curve with a scaled, piecewise segmented ``modification curve'' a la the Koch snowflake.


\subsection{Modification of Horizontal Lines}\label{modification}

In this subsection we define the ``modification curves'' that will be used to replace each horizontal segment piece in our inductive construction.
Before beginning, we observe that by the BCH formula (Theorem~\ref{t-BCH}) $\exp(X_i)$ commutes with both $\exp(X_4)$ and $\exp(X_5)$ for $i=1,\dots,5$. In particular for any $x_1,\dots,x_5 \in \mathbb{R}$,
\eqref{e-secondexp} gives
\begin{equation}
    \label{e-nicegroupop}
    \begin{aligned}
        (x_1,x_2,& x_3,0,0) * (0,0,0,x_4,x_5)\\
        &= (\exp(x_3X_3)*\exp(x_2X_2)*\exp(x_1X_1))*(\exp(x_5X_5) * \exp(x_4X_4)) \\
        &=
        \exp(x_5X_5) * \exp(x_4X_4) * \exp(x_3X_3)*\exp(x_2X_2)*\exp(x_1X_1)\\
        &=(x_1,x_2,x_3,x_4,x_5).
    \end{aligned}
\end{equation}
A similar argument gives $(x_1,x_2,x_3,x_4,x_5)=(0,0,0,x_4,x_5)*(x_1,x_2,x_3,0,0)$.

Throughout this subsection 
we fix the following parameters:
\begin{itemize}
    \item an interval $[a,b] \subset \mathbb{R}$ of length $L:=b-a\leq 1$,
    \item a constant $1 \leq \lambda<\frac32$,
    \item an integer $Q\geq 5$. 
\end{itemize}
Intuitively, $[a,b]$ represents the interval parameterizing a horizontal segment to be modified,
$\lambda$ is the speed of this segment, and $Q$ indexes both the number of pieces into which this segment will be divided as well as the size of the resulting perturbations in the $x_4$- and $x_5$-directions.
We also define
\begin{equation}\label{lambda'}
\lambda':=\left(1+\frac{2}{3Q} \right) \lambda
\end{equation} 
and note $\lambda' < \tfrac{17}{10}$.
This represents the speed of the new curve after modification.


Begin the construction by partitioning $[a,b]$ into $3Q+2$ consecutive intervals $I_i$ of equal length 
as follows:
\[ I_{i}=[a_i,b_i] := \left[ a+(i-1)\frac{L}{3Q+2}, a+i\frac{L}{3Q+2}\right], \qquad \text{for } 1\leq i\leq 3Q+2.\]
We point out that for each $i \in \{1,\dots,3Q+2\}$, 
\[
m(I_{i})=\frac{L}{3Q+2}\leq \frac{L}{17}.
\]
On the other hand, if $A$ is the union of any $Q$ such intervals,
\begin{equation}\label{Qintervals}
m(A)=\frac{QL}{3Q+2}=\frac{L}{3+\frac{2}{Q}}\geq \frac{L}{3+2}= \frac{L}{5}.
\end{equation}

\begin{construction}
    \label{c-construction}   

Write $\mu := \frac{\lambda L}{24Q}$, and suppose that $\zeta = \zeta_{\mu}$ is the ``staircase'' curve  connecting the origin to $\left(0,0,0,0,\mu^{3}\right)$ defined in Proposition~\ref{p-staircase}.
Note that the curve $t\mapsto \zeta\left(\frac{8(t-a_{Q+1})}{b_{Q+1}-a_{Q+1}}\right)$ defined on $I_{Q+1}$ is a concatenation of 8 horizontal $\lambda'$-segments
by \eqref{e-nicederivativeform} since
$$
\mu \cdot
\tfrac{8}{b_{Q+1} - a_{Q+1}}
=
\tfrac{\lambda L}{24Q}\cdot 8 \cdot \tfrac{3Q+2}{L} = \lambda',
$$
Similarly, $t\mapsto \zeta\left(\frac{8(b_{2Q+2}-t)}{b_{2Q+2}-a_{2Q+2}}\right)$ defined on $I_{2Q+2}$ is a concatenation of 8 horizontal $\lambda'$-segments starting at $\left(0,0,0,0,\mu^{3}\right)$ and ending at the origin.

Define $\alpha^+\colon [a,b] \to \mathbb{F}$ as follows:
\[
\alpha^+(t) = 
\begin{cases}
    \exp((t-a)\lambda'X_1) &\text{if } t \in I_{1}\cup \cdots \cup I_{Q}\\
    \left(\tfrac{\lambda L}{3},0, 0,0, 0\right) * \zeta\left(\frac{8(t-a_{Q+1})}{b_{Q+1}-a_{Q+1}}\right) 
    &\text{if } t \in I_{Q+1}\\
    \left(\tfrac{\lambda L}{3},0, 0,0, \mu^{3} \right) 
    *
    \exp((t-a_{Q+2})\lambda'X_1) &\text{if } t \in I_{Q+2}\cup \cdots \cup I_{2Q+1}\\
    \left(\tfrac{2\lambda L}{3},0, 0,0, 0\right) * \zeta\left(\frac{8(b_{2Q+2}-t)}{b_{2Q+2}-a_{2Q+2}}\right) 
    &\text{if } t \in I_{2Q+2}\\
    \left(\tfrac{2\lambda L}{3},0, 0,0, 0\right) * \exp((t-a_{2Q+3})\lambda'X_1) &\text{if } t \in I_{2Q+3}\cup \cdots \cup I_{3Q+2}
\end{cases}
\]
Let us parse this ``pedestrian overpass'' construction intuitively.
\begin{itemize}
\item (Along a footpath) On the first long interval $I_{1}\cup \cdots \cup I_{Q}$, the curve $\alpha^+$ is simply a horizontal $\lambda'X_1$-segment starting at the origin i.e. it is a Euclidean line segment from the origin to $\left( \tfrac{\lambda'QL}{3Q+2},0,0,0,0\right) = \left(\tfrac{\lambda L}{3},0, 0,0, 0\right)$.
\item (Up 8 flights of stairs) On $I_{Q+1}$, the curve then traverses 8 horizontal $\lambda'$-segments, and \eqref{e-nicegroupop} guarantees that this ends at the point $\left(\tfrac{\lambda L}{3},0, 0,0, \mu^{3} \right)$.
\item (Across the overpass) The curve then follows one horizontal $\lambda'X_1$-segment on the long middle interval $I_{Q+2}\cup \cdots \cup I_{2Q+1}$.
According to \eqref{e-nicegroupop}, this is once again a Euclidean line segment, terminating this time at $\left(\tfrac{2\lambda L}{3},0, 0,0, \mu^{3} \right) $.
\item (Down 8 flights of stairs) On $I_{2Q+2}$, the curve again traverses 8 horizontal $\lambda'$-segments
(this time in reverse) from $\left(\tfrac{2\lambda L}{3},0, 0,0, \mu^{3} \right) $ to $\left(\tfrac{2\lambda L}{3},0, 0,0, 0 \right)$. 
\item (Along a footpath) The curve is again a single horizontal $\lambda'X_1$-segment on the final long interval $I_{2Q+3}\cup \cdots \cup I_{3Q+2}$, 
and again one may apply \eqref{e-nicegroupop} to see that this is a Euclidean line segment terminating at $(\lambda L,0, 0,0, 0)$.
\end{itemize}

We then define the curve $\alpha^- \colon [a,b]\to \mathbb{F}$ to be $\alpha^- (t)=(-\lambda L,0,0,0,0)*\alpha^+ (a+b-t)$ and note that $\alpha^- (a)=0$ and $\alpha^- (b)=(-\lambda L,0,0,0,0)$.

Let us now define $\beta^+$.
Suppose that $\xi = \zeta_{- \mu}$ is the curve from Proposition~\ref{p-staircase} connecting the origin to $\left(0,0,0,-\mu^{3},0\right)$.
Define $\beta^+\colon [a,b] \to \mathbb{F}$ as follows:
$$
\beta^+(t) = 
\begin{cases}
    \exp(\lambda'(t-a)X_2) &\text{if } t \in I_{1}\cup \cdots \cup I_{Q}\\
    \left(0,\tfrac{\lambda L}{3}, 0,0, 0\right) * \xi\left(\frac{8(t-a_{Q+1})}{b_{Q+1}-a_{Q+1}}\right) 
    &\text{if } t \in I_{Q+1}\\
    \left(0,\tfrac{\lambda L}{3},0, -\mu^{3},0 \right) 
    *
    \exp(\lambda'(t-a_{Q+2})X_2) &\text{if } t \in I_{Q+2}\cup \cdots \cup I_{2Q+1}\\
    \left(0,\tfrac{2\lambda L}{3}, 0,0, 0\right) * \xi\left(\frac{8(b_{2Q+2}-t)}{b_{2Q+2}-a_{2Q+2}}\right) 
    &\text{if } t \in I_{2Q+2}\\
    \left(0,\tfrac{2\lambda L}{3}, 0,0, 0\right) * \exp( \lambda' (t-a_{2Q+3})X_2) &\text{if } t \in I_{2Q+3}\cup \cdots \cup I_{3Q+2}
\end{cases}
$$
We can once again understand this curve using the ``pedestrian overpass'' intuition above where motion along $X_1$ is replaced by motion along $X_2$ 
and vertical shifts in the $x_5$ coordinate are replaced by vertical shifts in the $x_4$ coordinate in the opposite direction.

Finally, define the curve $\beta^- \colon [a,b]\to \mathbb{F}$ given by $\beta^- (t)=(0,-\lambda L,0,0,0) *\beta^+(a+b-t)$, and note that $\beta^- (a)=0$ and $\beta^- (b)=(0,-\lambda L,0,0,0)$.

\end{construction}

Let us now gather some useful properties of these curves. 

\begin{proposition}\label{p-distances}
The curves $\alpha^+$, $\alpha^-$, $\beta^+$, and $\beta^-$ defined in Construction~\ref{c-construction} have the following structure:
\begin{enumerate}
    \item \label{enum-easy1} $\alpha^\pm (a) = \beta^\pm (a) = (0,0,0,0,0)$.
    \item \label{enum-easy2}$\alpha^\pm(b) = (\pm \lambda (b-a),0,0,0,0)$ and $\beta^\pm(b) = (0,\pm \lambda (b-a), 0, 0, 0)$.
    \item \label{enum-derivative} Fix $1\leq i\leq 3Q+2$ with $i\neq Q+1$ and $i\neq 2Q+2$. Then $\alpha^{\pm}|_{I_{i}}$ is a horizontal $\pm \lambda' X_{1}$-segment and $\beta^{\pm}|_{I_{i}}$ is a horizontal $\pm \lambda'X_{2}$ segment.
    \item \label{i-smallpartition} Suppose $\{J_j\}_{j=1}^N$ is a partition of $[a,b]$ into $N$ consecutive intervals of equal length where $N$ is a positive multiple of $8(3Q+2)$. Then for every $j \in \{1,\dots,N\}$, each of $\alpha^\pm|_{J_j}$ and $\beta^\pm|_{J_j}$ is a 
    horizontal $\lambda'$-segment.
\end{enumerate}
\end{proposition}

\begin{proof} 
Properties (\ref{enum-easy1}),(\ref{enum-easy2}) and (\ref{enum-derivative}) follow immediately from Construction \ref{c-construction}. To see (\ref{i-smallpartition}), note that partitioning each interval $I_i$ into 8 subintervals of equal length ensures that the curves $\alpha^\pm$ and $\beta^\pm$ are equal on each such subinterval to a horizontal $\lambda'$-segment.
Therefore, since $N$ is a multiple of $3Q+2$, partitioning the full interval $[a,b]$ into $N$ consecutive intervals $J_j$ of equal length 
also partitions each $I_i$ into $N/(3Q+2)$ intervals of equal length.
Since $N/(3Q+2)$ is a multiple of 8, this verifies (\ref{i-smallpartition}).
\end{proof}

\begin{proposition}\label{constructiondistances}
The curves $\alpha^+$, $\alpha^-$, $\beta^+$, and $\beta^-$ defined in Construction~\ref{c-construction} satisfy the following distance estimates:
\begin{enumerate}
    \item \label{enum-compact} $d_{c}(\alpha^{\pm}(t),0)\leq \tfrac{17}{10}$ 
    and $d_{c}(\beta^{\pm}(t),0)\leq \tfrac{17}{10}$
    for all $t\in [a,b]$.
    \item \label{enum-euclideanbound} $d_{c}(\exp(\pm  \lambda (t-a)  X_1),\alpha^\pm (t))\leq \frac{2(b-a)}{Q}$ and $d_{c}(\exp(\pm  \lambda (t-a) X_2) , \beta^\pm (t))
    \leq \frac{2(b-a)}{Q}$ for all $t\in [a,b]$.
\end{enumerate}
\end{proposition}

This second property shows that the ``modification curves'' $\alpha^\pm$ and $\beta^\pm$  do not deviate too far from the horizontal segments they modify.

\begin{proof}
It suffices to describe the case of the curve $\alpha^+$, since the other cases are similar. 
To see (\ref{enum-compact}), 
notice that $\alpha^+$ is simply a concatenation of 
horizontal $\lambda'$-segments.
Thus
\eqref{e-length} implies
$\ell(\alpha^+) = \lambda' L$,
and hence the assumptions $Q \geq 5$, $\lambda < \tfrac32$, and $L \leq 1$ imply
\[
d_{c}(\alpha^+(t),0)
\leq \ell(\alpha^{+})
\leq
\lambda' L
=(1+\tfrac{2}{3Q})\lambda L
<
\tfrac{17}{10}. 
\]



We now prove (\ref{enum-euclideanbound}). 
To simplify notation, assume without loss of generality that $a=0$ and $b=L$. 
Recall $\lambda < \tfrac32$ and $\lambda'=(1+\tfrac{2}{3Q})\lambda$
so that $\lambda' - \lambda = \tfrac{2 \lambda}{3Q}$. 

By Remark~\ref{rem-distance}, 
we have for all $t\in I_{1}\cup \cdots \cup I_{Q}$ that
\begin{align}\label{IQest}
d_{c}(\alpha^+(t),\exp(\lambda t X_{1})) = 
d_{c}(\exp(\lambda' t X_{1}),\exp(\lambda t X_{1}))
=
\left| t\lambda' -t\lambda\right|
= \tfrac{2 \lambda t}{3Q}
< \tfrac{L}{Q}.
\end{align}

Next suppose $t\in I_{Q+1}= [a_{Q+1},b_{Q+1}]$.
As above, by \eqref{e-length}
\[
d_{c}\left(\alpha^+(t),\alpha^+\left( a_{Q+1}
\right)\right)
\leq
\ell(\alpha^+|_{I_{Q+1}})
=
\lambda'\tfrac{L}{3Q+2} = \tfrac{\lambda L}{3Q}
<
\tfrac{L}{2Q}.\]
Similarly, 
since $m(I_{Q+1})=\tfrac{L}{3Q+2}$, 
\[
d_{c}\left(\exp \left( \lambda  a_{Q+1}
X_{1} \right),\exp(\lambda tX_{1})\right)
\leq 
\lambda \tfrac{L}{3Q+2} < \tfrac{L}{2Q}.
\]
Since  $a_{Q+1} = b_Q$ is also an endpoint of $I_{Q}$, \eqref{IQest} gives
\[ 
d_{c}\left(\alpha^+ \left( a_{Q+1} 
\right), \exp \left(\lambda  a_{Q+1}
X_{1} \right) \right)
=
d_{c}\left(\alpha^+ \left(b_Q \right), \exp \left(\lambda b_Q X_{1} \right) \right)
< \tfrac{L}{Q}.
\]
Combining the previous three estimates, we have for all $t\in I_{Q+1}$
\[ 
d_{c}(\alpha^+(t),\exp(\lambda tX_{1}))
\leq \tfrac{2L}{Q}.
\]

Next suppose $t\in I_{Q+2}\cup \cdots \cup I_{2Q+1}$. Recall that $\alpha^+|_{I_{Q+2}\cup \cdots \cup I_{2Q+1}}$ is a horizontal $\lambda' X_1$-segment starting at $\alpha^+\left( a_{Q+2}
\right)=
\left(\tfrac{\lambda L}{3},0, 0,0, \left( \frac{\lambda L}{24Q} \right)^{3}\right)$.
Using \eqref{e-nicegroupop} and \eqref{e-groupoperation},
\begin{align*}
\alpha^+(t)
&=
\left(\tfrac{\lambda L}{3},0, 0,0,\left(\tfrac{\lambda L}{24Q}\right)^{3}\right)
*
\exp\left(\lambda' \left(t- a_{Q+2}
\right) X_1 \right)\\
&=
\left(\tfrac{\lambda L}{3}+\lambda' \left(t-
 a_{Q+2}
\right),0,0,0, \left(\tfrac{\lambda L}{24Q} \right)^{3} \right).
\end{align*}
Therefore, 
temporarily writing $t' := \tfrac{\lambda L}{3}+\lambda' \left(t-
 a_{Q+2}
\right)$
where
$a_{Q+2} = \tfrac{(Q+1)L}{3Q+2}$
gives
\begin{align*}
d_{c}(\alpha^+(t),\exp(\lambda tX_{1}))
&\leq
d_{c}\left(\alpha^+(t),\exp\left(t'X_1\right)\right)
+
d_{c}(\exp(t'X_1),\exp(\lambda tX_{1}))
\\
&=d_{c}\left(\left(0,0,0,0, \left( \tfrac{\lambda L}{24Q} \right)^3\right),0\right) 
+
\left|t' -  \lambda t \right|\\
&\leq 
8\left(\frac{\lambda L}{24Q} \right) + t(\lambda' - \lambda) + \left|\frac{\lambda L}{3} - \lambda' 
 a_{Q+2}
\right| 
\\
&\leq 
\frac{\lambda L}{3Q}
+
\frac{L}{Q} + \left|\frac{\lambda L}{3}-\frac{(Q+1)\lambda L}{3Q} \right|\\
&= \frac{2\lambda L}{3Q} + \frac{L}{Q}
\leq \frac{2L}{Q}.
\end{align*}
The equality in the second line above follows directly from the left invariance of the metric (after multiplying both terms by $\exp(t'X_1)^{-1} = (-t',0,0,0,0)$), the commutativity in \eqref{e-nicegroupop},
and Remark~\ref{rem-distance}.
The third line follows from Corollary~\ref{c-verticalpath}.

The case $t\in I_{2Q+3}\cup \cdots \cup I_{3Q+2}$
follows as in \eqref{IQest} since
\[
\tfrac{2 \lambda L}{3} + (t-a_{2Q+3})\lambda' = 
\tfrac{2 \lambda L}{3} + \lambda' t - \tfrac{(2Q+2)\lambda'L}{3Q+2}
=
\lambda' t - \tfrac{2\lambda L}{3Q},
\]
and so
\begin{equation}
    \label{e-symmetricarg}
\begin{aligned}
d_{c}(\alpha^+(t),\exp(\lambda t X_{1})) &= 
d_{c}\left(\exp\left(\left(\lambda' t - \tfrac{2\lambda L}{3Q}\right) X_{1}\right),\exp(\lambda t X_{1})\right)\\
&=
\left| \lambda' t - \tfrac{2\lambda L}{3Q} - \lambda t \right|
=
\tfrac{2\lambda}{3Q}(L-t)
< \tfrac{L}{Q}.
\end{aligned}
\end{equation}

Finally, if $t\in I_{2Q+2}$, we have as before
\[
d_{c}\left(\alpha^+(t),\alpha^+\left(b_{2Q+2}  \right)\right)
<
\tfrac{L}{2Q}
\quad
\text{ and }
\quad
d_{c}\left(\exp \left( \lambda b_{2Q+2}X_{1} \right),\exp(\lambda tX_{1})\right)
< \tfrac{L}{2Q}.
\]
Since $b_{2Q+2} = a_{2Q+3}$, \eqref{e-symmetricarg} gives
\[
d_c(\alpha^+(b_{2Q+2}),\exp(\lambda b_{2Q+2} X_1) < \tfrac{L}{Q}.
\]
Combining these three estimates completes the proof of (\ref{enum-euclideanbound}).
\end{proof}

For the remainder of the paper, fix a constant so that 
\eqref{disugdist} holds for $K = \overline{B(0,2)}$. In particular, choose an integer $\kappa \geq 1$ such that
\begin{equation}\label{distancecomparison}
 |x-y|\leq \kappa d_{c}(x,y) \mbox{ for all }x,y\in \overline{B(0,2)}.
\end{equation}

In the following proposition we show that any horizontal curve that is close to a ``modifying curve'' from Construction~\ref{c-construction} cannot be too well approximated by a $C^{1}_H$ curve in the Lusin sense when 
one of the two horizontal derivatives of the $C^{1}$ curve stay bounded well away from zero. The idea of the proof is that such an assumption forces particular behavior of the vertical coordinates of the $C^{1}_{H}$ curve by Lemma \ref{unreachable}. This then implies it cannot coincide with the given curve in more than one long interval.

\begin{proposition}\label{badintersect}
Suppose $\rho$ is one of the curves $\alpha^+, \alpha^-, \beta^+, \beta^-$ defined in Construction~\ref{c-construction}.
Let $\eta\colon [a,b]\to \mathbb{F}$ be any horizontal curve with $d_{c}(\eta(t),\rho(t))< \tfrac{1}{10\kappa} \left(\frac{\lambda (b-a)}{24Q}\right)^{3}$ for all $t\in [a,b]$.
Suppose $\Gamma\colon [a,b]\to \mathbb{F}$ is any $C^{1}_H$ curve such that $|\Gamma_1'| \geq \tfrac12$ on $[a,b]$ or $|\Gamma_2'| \geq \tfrac12$ on $[a,b]$.
Then
\begin{equation}
\label{e-overlapProp}
m\{t\in [a,b]:\Gamma(t)=\eta(t)\} \leq \tfrac45 (b-a).
\end{equation}
\end{proposition}

\begin{proof}
Let $\epsilon := \left(\frac{\lambda (b-a)}{24Q}\right)^{3} < 1$. 
Since $\kappa\geq 1$ it follows that $d_{c}(\eta(t),\rho(t))<1/10$. By combining this with (1) in Proposition~\ref{constructiondistances} we note that for $t\in [a,b]$
\begin{equation}\label{etaintheball}
d_{c}(\eta(t),0)\leq d_{c}(\eta(t),\rho(t))+d_{c}(\rho(t),0)\leq \tfrac{1}{10}+\tfrac{17}{10}<2.
\end{equation}
This together with \eqref{distancecomparison} implies
\begin{equation}
    \label{e-etacloseEuclidean}
    |\eta(t)-\rho(t)| \leq \kappa d_c(\eta(t),\rho(t)) < \tfrac{\varepsilon}{10}
    \quad
    \text{for all } t \in [a,b].
\end{equation}

We will first prove the proposition for the curve $\rho=\alpha^+$. 
Suppose that $\Gamma'_1 \geq \tfrac12$ on $[a,b]$.
If either $\Gamma(t)\neq \eta(t)$ for all $t\in I_{1}\cup \cdots \cup I_{Q}$
or
$\Gamma(t)\neq \eta(t)$ for all  $t\in I_{Q+2}\cup \cdots \cup I_{2Q+1}$, then by \eqref{Qintervals}
\begin{equation}
    \label{e-smallMeasure}
m\{t\in [a,b]:\Gamma(t)\neq \eta(t)\} \geq 
\tfrac15 (b-a)
\end{equation}
and the claim is proven.

Assume instead that $\Gamma(s)= \eta(s)$ and $\Gamma(t)= \eta(t)$
for some $s \in I_{1}\cup \cdots \cup I_{Q}$ and $t\in I_{Q+2}\cup \cdots \cup I_{2Q+1}$. 
We will now see that this assumption leads to a contradiction.
We first observe from 
\eqref{e-etacloseEuclidean} and the facts that
$\alpha^+_5(s)=0$ and $\alpha^+_5(t)=  \epsilon$ that
\begin{align*}
    \eta_5(t)-\eta_5(s)
    &=
    \alpha^+_5(t) + \eta_5(t) - \alpha^+_5(t) + \alpha^+_5(s) -\eta_5(s)\\
    &\geq 
    \alpha^+_5(t) - |\eta_5(t)-\alpha^+_5(t)| - |\eta_5(s)-\alpha^+_5(s)|
    >
    \epsilon
    -
    2(\tfrac{\epsilon}{10})
    =
    \tfrac{4 \epsilon}{5}.
\end{align*}
On the other hand, Lemma~\ref{unreachable}(1) gives
\begin{equation}
    \label{e-Lemma34applied1}
    \begin{aligned}
\eta_5(t)-\eta_5(s)
    = \Gamma_5(t) -\Gamma_5(s)
    &\leq \tfrac12(\Gamma_1(t)\Gamma_2(t)^2
    -
    \Gamma_1(s)\Gamma_2(s)^2)\\
        &\leq
    \tfrac12|\eta_1(t)|\eta_2(t)^2
    +
    \tfrac12|\eta_1(s)|\eta_2(s)^2.
\end{aligned}
\end{equation}
Since $\alpha^+_2(t)=\alpha^+_2(s)=0$,
\eqref{e-etacloseEuclidean} implies that
$|\eta_2(t)| < \tfrac{\varepsilon}{10}$
and
$|\eta_2(s)| < \tfrac{\varepsilon}{10}$.
Note also that
$$
|\eta_1(t)|
\leq 
|\eta_1(t) - \alpha^+_1(t)| + |\alpha^+_1(t)|
\leq \tfrac{\epsilon}{10} + \tfrac{17}{10} < 2
$$
since 
$|\alpha_1^+(t)| \leq \int_0^t |(\alpha^+)'| \leq \lambda' (b-a) < \tfrac{17}{10}$.
Similarly we have $|\eta_1(s)|<2$. 
Therefore, 
\begin{align*}
    \tfrac{4\epsilon}{5}
    <\eta_5(t)-\eta_5(s)
    \leq
    \tfrac12|\eta_1(t)|\eta_2(t)^2
    +
    \tfrac12|\eta_1(s)|\eta_2(s)^2
    <
    2(\tfrac{\epsilon}{10})^2< \tfrac{\epsilon}{50}.
\end{align*}
This is impossible, so \eqref{e-smallMeasure} holds.

Now, if $\Gamma_1' \leq -\tfrac12$, the argument is similar, and it suffices to show that assuming
$\Gamma(s)= \eta(s)$ and $\Gamma(t)= \eta(t)$
for some $s \in I_{Q+2}\cup \cdots \cup I_{2Q+1}$ and $t\in I_{2Q+3}\cup \cdots \cup I_{3Q+2}$ leads to a contradiction.
This time, $\alpha^+_5(s)=\epsilon$ and $\alpha^+_5(t)=  0$, so
\begin{equation*}
    \eta_5(s)-\eta_5(t)
    \geq 
    \alpha^+_5(s) - |\eta_5(s)-\alpha^+_5(s)| - |\eta_5(t)-\alpha^+_5(t)|
    >
    \tfrac{4 \epsilon}{5}
\end{equation*}
while (2) from Lemma~\ref{unreachable} implies
\begin{align*}
\tfrac{4\varepsilon}{5} < \eta_5(s)-\eta_5(t)
    = -(\Gamma_5(t) -\Gamma_5(s))
    &\leq  -\tfrac12(\Gamma_1(t)\Gamma_2(t)^2
    -
    \Gamma_1(s)\Gamma_2(s)^2)\\
        &\leq
    \tfrac12|\eta_1(t)|\eta_2(t)^2
    +
    \tfrac12|\eta_1(s)|\eta_2(s)^2
    <
    \tfrac{\varepsilon}{50},
\end{align*}
which is impossible.

Suppose now that $|\Gamma'_2| \geq \tfrac12$ on $[a,b]$.
According to the definition of $\alpha^+$, its second coordinate $\alpha_2^+$ is non-zero only on $I_{Q+1}$ and $I_{2Q+2}$.
Moreover, $\alpha_2^+=0$ at the endpoints of these segments, and $|(\alpha^+)'| \leq \lambda'$. Since these intervals are disjoint, it follows that 
$|\alpha_2^+| \leq \lambda' \tfrac{b-a}{3Q+2} = \lambda \tfrac{b-a}{3Q}$ on $[a,b]$.
Combining with \eqref{e-etacloseEuclidean},
this implies that $|\eta_{2}| < \tfrac{\varepsilon}{10} + \lambda \tfrac{b-a}{3Q} < \lambda \tfrac{b-a}{2Q}$ on $[a,b]$. Hence 
\[m\{t\in [a,b]:\Gamma(t)=\eta(t)\} \leq 
m\left\{t\in [a,b]: |\Gamma_{2}(t)| \leq \lambda \tfrac{b-a}{2Q}\right\}.\]
If $\Gamma_2'\geq \tfrac12$ on $[a,b]$, then $\Gamma_{2}(t)-\Gamma_{2}(s)=\int_{s}^{t}\Gamma_{2}'\geq \tfrac12 (t-s)$ 
and if $\Gamma_2'\leq -\tfrac12$ then 
$\Gamma_{2}(s)-\Gamma_{2}(t)=-\int_{s}^{t}\Gamma_{2}'\geq \tfrac12 (t-s)$ for all $t\geq s$.
Thus $|t-s|\leq 2|\Gamma_{2}(t)-\Gamma_{2}(s)|$ for all $s,t \in [a,b]$. Hence if $|\Gamma_{2}(t)| \leq \lambda \tfrac{b-a}{2Q}$ and $|\Gamma_{2}(s)| \leq \lambda \tfrac{b-a}{2Q}$ for some $s,t \in [a,b]$, then $|t-s|\leq 2\lambda \frac{b-a}{Q}$. 
Applying the assumptions $Q\geq 5$ and $\lambda<3/2$ gives
\[ m \left\{t\in [a,b]:|\Gamma_{2}(t)| \leq \lambda \tfrac{b-a}{2Q}\right\} \leq 2\lambda \tfrac{b-a}{Q} < \tfrac35 (b-a).
\]
This completes the proof in the case $\rho=\alpha^{+}$.

Suppose $\rho=\alpha^-$. If $\Gamma\colon [a,b] \to \mathbb{F}$ is $C^1_H$, define for $t$ in $[a,b]$,
\[
\hat{\Gamma}(t) :=(\lambda L,0,0,0,0)*\Gamma(a+b-t).
\]
Notice that if $\Gamma$ satisfies the conditions in the statement of the proposition then so does $\hat{\Gamma}$. Further, if  $\Gamma(t) = \alpha^-(t)$ for some $t \in [a,b]$, then it follows from the definition of $\alpha^{-}$ that $\hat{\Gamma}(t)= \alpha^+(t)$. Hence this case follows from the case $\rho = \alpha^+$.

Now suppose that $\rho = \beta^+$. 
The proof in the case $|\Gamma_1'| \geq \tfrac12$ is nearly identical to that of the above case when $\rho = \alpha^+$ and $|\Gamma_2'| \geq \tfrac12$, so we leave the details to the reader.

The proof in the case $|\Gamma_2'| \geq \tfrac12$ is also similar to (and simpler than) that of the above case when $|\Gamma_1'| \geq \tfrac12$, and we outline the main differences here.
Suppose first that $\Gamma_2' \geq \tfrac12$.
As before, we need only verify that it is impossible for 
$\Gamma(s)= \eta(s)$ and $\Gamma(t)= \eta(t)$ for some $s \in I_{1}\cup \cdots \cup I_{Q}$  and $t\in I_{Q+2}\cup \cdots \cup I_{2Q+1}$ simultaneously.

Since, in this case, $\beta^+_4(s)=0$ and $\beta^+_4(t) = -\epsilon$, we have
\begin{align*}
    \eta_4(t)-\eta_4(s)
    \leq 
    \beta^+_4(t) + |\eta_4(t)-\beta^+_4(t)| + |\beta^+_4(s)-\eta_4(s)| 
    <
    -\epsilon
    +
    2(\tfrac{\epsilon}{10})
    =
    -\tfrac{4 \epsilon}{5}.
\end{align*}
On the other hand, applying (3) from Lemma~\ref{unreachable} replaces \eqref{e-Lemma34applied1} with 
\begin{align*}
    \eta_4(t)-\eta_4(s) = \Gamma_{4}(t)-\Gamma_{4}(s) \geq 0,
\end{align*}
and we have already reached a contradiction.

When $\Gamma_2' \leq -\tfrac12$, the proof follows similarly to the above argument as it did in the case $\rho = \alpha^+$.
Since the case $\rho = \beta^-$ can be concluded from the previous case as above, this completes the proof.
\end{proof}





\subsection{A Sequence of Horizontal Curves}\label{sequenceofcurves}

We are now ready to construct the sequence of horizontal curves $\gamma_{n}\colon [0,1]\to \mathbb{F}$. We first give the underlying idea. We start with the horizontal segment $\gamma_{1}(t):=\exp(tX_{1})$. At each subsequent step, $\gamma_{n}$ will be a horizontal curve defined on $[0,1]$ with the property that, when $[0,1]$ is subdivided into an appropriate number $N_{n}$
of subintervals, the restriction of $\gamma_n$ to each such subinterval is a horizontal segment with direction equal to one of $\pm X_{1}, \pm X_{2}$ with speed $\lambda_{n}$. To construct $\gamma_{n+1}$ from $\gamma_{n}$, we will apply Construction~\ref{c-construction} to modify these horizontal segments with parameters $\lambda=\lambda_{n}$ and $Q=5^{n}$. 


We now give the details of the construction.
Set $\lambda_1 = 1$ and $\lambda_{n+1}=(1+\frac{2}{3\cdot 5^{n}})\lambda_{n}$ for $n\geq 1$. This yields $\lambda_n = \prod_{k=1}^{n-1} \left(1+\frac{2}{3\cdot 5^{k}}\right)$ for $n \geq 2$. Since $1+x \leq e^x$ when $x>0$, we have
\begin{equation}\label{lambdabound}
1 \leq \lambda_n 
\leq
\exp\left( \sum_{k=1}^{n-1} \frac{2}{3\cdot 5^{k}} \right)
< 
e^{1/6} < \tfrac32
\end{equation}
for all $n\geq 2$.
Recall the constant $\kappa$ chosen in \eqref{distancecomparison}.
Set $N_1 = 1$ and, for $n\geq 1$,
write 
$$
N_{n +1}
= 
80\kappa(3\cdot 5^{n}+2)(24 \cdot 5^nN_n)^3.
$$
Note in particular that $\{N_n\}_{n\in\mathbb{N}}$ is an increasing sequence and that $\tfrac{N_{n+1}}{N_n}$ is a multiple of $8(3\cdot 5^{n}+2)$. Moreover, since $\lambda_{n}\geq 1$,
\begin{equation}
    \label{e-sillyNnbound}
    N_{n+1}^{-1} \leq \frac{1}{10\kappa}\left( \frac{\lambda_{n}}{24 \cdot 5^nN_n }\right)^{3}.
\end{equation}

For each $n \geq 1$, let $\{J^n_j\}_{j=1}^{N_n}$ be a partition of $[0,1]$ into $N_n$ consecutive intervals of equal length. Note that, for each $j \in \{1,\dots,N_n\}$, the interval $J_j^n$ is equal to the union of $\tfrac{N_{n+1}}{N_n}$ intervals of the form $J_k^{n+1}$ for consecutive indices $k \in \{1,\dots,N_{n+1}\}$.
Conversely, if $n\geq 2$, then $J_j^n$ is one of the intervals obtained when $J_i^{n-1}$ is partitioned into $\tfrac{N_{n}}{N_{n-1}}$ intervals of equal length for some $i \in \{1,\dots,N_{n-1}\}$.

\begin{construction}
    \label{c-inductivedefinition}
We now inductively define horizontal curves $\gamma_n\colon [0,1] \to \mathbb{F}$ for every $n\geq 1$ that have the following property (among others to be studied later):
\begin{equation}\label{inductiveproperty}
\gamma_n|_{J^n_j} 
\mbox{ is a horizontal }\lambda_n \mbox{-segment}
\mbox{ for }
j \in \{1,\dots,N_n\}.
\end{equation}

The first curve $\gamma_{1}\colon [0,1]\to \mathbb{F}$ is defined as $\gamma_{1}(t):=\exp(tX_{1})$. Recalling $N_{1}=1$ and $\lambda_{1}=1$, clearly \eqref{inductiveproperty} holds for $n=1$.

For the inductive step, fix $n\geq 1$ for which a horizontal curve $\gamma_{n}\colon [0,1] \to \mathbb{F}$ has been defined such that \eqref{inductiveproperty} holds. Temporarily fix 
$j \in \{1,\dots,N_n\}$
and write $[a,b]:=J^n_j$. Then \eqref{inductiveproperty} implies that $\gamma_{n}(t) = \gamma_{n}(a) *\exp( \lambda_n (t-a) V)$ for 
all $t \in [a,b]$ and
some $V\in \{\pm X_{1}, \pm X_{2}\}$. We will now ``modify'' this curve as discussed above
based on the choice of $V$: 
apply Construction~\ref{c-construction} with parameters $[a,b]=J_j^n$, $\lambda=\lambda_{n}$, and $Q=5^{n}$
to 
define a curve $\rho \colon [a,b]\to \mathbb{F}$
such that 
$\rho=\alpha^\pm$ if $V = \pm X_1$ or $\rho=\beta^\pm$ if $V = \pm X_{2}$.
In particular, $\rho$ is a horizontal curve with the same endpoints as $\exp( \lambda_n (t-a)V)$ by (\ref{enum-easy1}) and (\ref{enum-easy2}) of Proposition~\ref{p-distances}. We then define 
\begin{equation}
    \label{e-nextstepwitheta}
    \gamma_{n+1}(t)=\gamma_{n}(a)*\rho(t) \quad \text{ for } t\in [a,b] = J_j^n.
\end{equation}
Repeating this process for all $j \in \{1,\dots,N_n\}$
defines $\gamma_{n+1}$ on $[0,1]$. Since each application of Construction \ref{c-construction} preserves the endpoints of the curve, the map $\gamma_{n+1}$ is continuous,
and it is horizontal since $\gamma_{n+1}|_{J_j^n}$ is horizontal for each $j$.

We now check that $\gamma_{n+1}$ satisfies \eqref{inductiveproperty}. 
Fix $j \in \{1,\dots,N_{n+1}\}$.
Then $J_{j}^{n+1} \subset J_{j_0}^n$ for some $j_0 \in \{1,\dots,N_{n}\}$,
and $J_{j}^{n+1}$ is one of the intervals obtained when $J_{j_0}^n$ is partitioned into $\tfrac{N_{n+1}}{N_n}$ consecutive intervals of equal length.
Since $\tfrac{N_{n+1}}{N_n}$ is a multiple of $8(3\cdot 5^{n}+2)$,
property~(\ref{i-smallpartition}) in Proposition~\ref{p-distances} with $\lambda' = \left(1+\frac{2}{3\cdot 5^{n}}\right)\lambda_{n} = \lambda_{n+1}$
implies that $\gamma_{n}(a)^{-1}*\gamma_{n+1}|_{J_j^{n+1}}$ is a 
horizontal $\lambda_{n+1}$-segment,
and thus so is $\gamma_{n+1}|_{J_j^{n+1}}$.
\end{construction}

The following are other properties of the curves $\gamma_{n}$ that arise as a consequence of Construction~\ref{c-inductivedefinition}.

\begin{proposition}
    \label{p-inductiveconstruction}
    For every $n \in \mathbb{N}$ the following properties hold:
    \begin{enumerate}
        \item \label{item-segments1} For every $j \in \{1,\dots,N_n\}$, the curve $(\gamma_n^{-1}(a)*\gamma_{n+1})|_{J_j^n}$ is one of the curves $\alpha^+$, $\alpha^-$, $\beta^+$, or $\beta^-$ from Construction~\ref{c-construction} with parameters $[a,b]=J_j^n$, $\lambda=\lambda_{n}$, and $Q=5^{n}$.
        \item \label{item-segments2} For every $j \in \{1,\dots,N_n\}$, there is some $V\in \{\pm X_{1}, \pm X_{2}\}$ such that $\gamma_n|_{J_j^n}$ is a horizontal $\lambda_{n}V$-segment. Moreover, suppose $\{I_{j,i}^n\}_{i=1}^{3 \cdot 5^{n}+2}$ is a partition of $J_j^{n}$ into $3 \cdot 5^{n}+2$ consecutive intervals of equal length. 
        Then, for each $1\leq i\leq 3 \cdot 5^{n}+2$ with $i\neq 5^{n}+1$ and $i\neq 2\cdot 5^{n}+2$, the curve $\gamma_{n+1}$ restricted to $I_{j,i}^n$ is a horizontal $\lambda_{n+1}V$-segment (for the same choice of $V$).
\item \label{item-Lip} $\gamma_n$ is $\lambda_n$-Lipschitz.
\item \label{item-cauchy} For all $t \in [0,1]$,
    \begin{equation}
    \label{e-shortsteptonext}
    d_{c}(\gamma_{n}(t),\gamma_{n+1}(t))
    \leq
    2N_{n}^{-1}5^{-n}.
\end{equation}
\end{enumerate}
\end{proposition}


\begin{proof}
Property (\ref{item-segments1}) follows immediately from the definition of $\gamma_{n+1}$ in Construction \ref{c-inductivedefinition}, in particular by \eqref{e-nextstepwitheta}. Property (\ref{item-segments2}) follows similarly from Proposition~\ref{p-distances}(\ref{enum-derivative}).

To prove (\ref{item-Lip}), fix $x,y \in [0,1]$ with $x \leq y$. Then $x \in J_{j_1}^{n}$ and $y \in J_{j_2}^{n}$ for some $j_1,j_2 \in \{1,\dots,N_{n}\}$ with $j_{1}\leq j_{2}$. Write $[a_i,b_i] := J_i^{n}$ for each $i$ between $j_1$ and $j_2$.
According to \eqref{inductiveproperty} and Remark~\ref{rem-distance} and using the fact $a_{i+1}=b_{i}$, we then have
\begin{align*}
    d_c(\gamma_{n}(x),\gamma_{n}(y)) &\leq d_c(\gamma_{n}(x),\gamma_{n}(b_{j_1}))+\sum_{i = j_1+1}^{j_2-1} d_c(\gamma_{n}(a_i),\gamma_{n}(b_i))\\ & \hspace{2in} + d_c(\gamma_{n}(a_{j_2}),\gamma_{n}(y))\\
    &= \lambda_{n}(b_{j_1}-x) + \sum_{i = j_1+1}^{j_2-1} \lambda_{n}(b_i-a_i) + \lambda_{n}(y-a_{j_2})\\
    &=\lambda_{n}(y-x).
\end{align*}

We finally verify property (\ref{item-cauchy}) separately on each interval $J_j^n$ for $1\leq j\leq N_{n}$. Fix such a $j$, write $[a,b] := J_j^n$, and fix $t \in [a,b]$. 
From \eqref{inductiveproperty}, we can write $\gamma_{n}(t)= \gamma_{n}(a)*\exp(\lambda_{n}(t-a)V)$ on $[a,b]$ for some $V\in \{\pm X_{1}, \pm X_{2}\}$. 
According to \eqref{e-nextstepwitheta}, $\gamma_{n+1}(t)=\gamma_{n}(a)*\rho(t)$, where $\rho$ is $\alpha^{\pm}$ or $\beta^{\pm}$ depending on whether $V=\pm X_{1}$ or $\pm X_{2}$ respectively. Since $J_j^n$ has length $N_{n}^{-1}$, we use Proposition~\ref{constructiondistances}(\ref{enum-euclideanbound}) to estimate as follows:  
\begin{align*}
d_{c}(\gamma_{n}(t),\gamma_{n+1}(t)) &= d_{c}( \exp(\lambda_{n}(t-a)V), \rho(t)  )
\leq \tfrac{2(b-a)}{5^{n}}
=2N_{n}^{-1}5^{-n}.
\end{align*}
This proves (\ref{item-cauchy}).
\end{proof}

\subsection{The Limit Curve}\label{limitcurve}

In this section we show that the sequence of curves $\gamma_{n}$ defined in Construction \ref{c-inductivedefinition} converges uniformly to a limiting curve $\gamma$ that does not overlap with any $C^{1}_H$ curve on a set of positive measure, hence proving Theorem \ref{lusinzero}.

First, recall that $\{\lambda_{n}\}_{n=1}^\infty$ is an increasing sequence and, by \eqref{lambdabound}, it takes values in $[1,\tfrac32)$. Hence we can define \[\lambda_0 := \lim_{n\to \infty}\lambda_{n} \in [1,\tfrac32].\]

\begin{proposition}\label{limitislipschitz}
The sequence of Lipschitz curves $\{\gamma_{n}\}_{n=1}^\infty$ converges uniformly in $\mathbb{F}$ to a $\lambda_0$-Lipschitz 
curve $\gamma\colon [0,1]\to \mathbb{F}$. In particular, we have for all $n \geq 1$ 
and $t \in [0,1]$ that
\begin{equation}
\label{e-unifconv}
d_{c} (\gamma(t),\gamma_{n}(t)) \leq 
2\sum_{k=n}^{\infty} N_{k}^{-1} 5^{-k}
\end{equation}
\end{proposition}

\begin{proof}
Using \eqref{e-shortsteptonext}, we have for all $m>n$
\begin{equation}\label{Cauchy}
d_{c} (\gamma_{m}(t),\gamma_{n}(t)) \leq 2\sum_{k=n}^{m-1} N_{k}^{-1}5^{-k} \mbox{ for all }t\in [0,1].
\end{equation}
Since $N_k \geq 1$ for all $k\geq 1$, the series $\sum_{k=1}^{\infty} N_{k}^{-1}5^{-k}$ converges. Thus, for all fixed $t\in [0,1]$, the sequence $\{\gamma_{n}(t)\}_{n=1}^\infty$ is a Cauchy sequence in the complete metric space $\mathbb{F}$ 
and hence converges to some point in $\mathbb{F}$ that we define to be $\gamma(t)$. Passing $m\to \infty$ in \eqref{Cauchy} then gives \eqref{e-unifconv}
for all $t \in [0,1]$.
Hence $\gamma_{n}$ converges uniformly to a map $\gamma\colon [0,1]\to \mathbb{F}$ with respect to $d_{c}$.

To see that $\gamma$ is $\lambda_0$-Lipschitz, recall that each curve $\gamma_n$ is $\lambda_n$-Lipschitz. 
Hence, for all $x,y \in [0,1]$, we have
    $$
    d_{c}(\gamma(x),\gamma(y)) = \lim_{n\to\infty} d_{c}(\gamma_n(x),\gamma_n(y)) \leq \lim_{n \to \infty} \lambda_n |x-y| = \lambda_0 |x-y|.
    $$
\end{proof}

Since, by Proposition \ref{limitislipschitz}, the curve $\gamma$ is Lipschitz with respect to the metric $d_{c}$, it follows that it is also a horizontal curve 
(see Remark~\ref{rem-horiz}). 
In particular its derivative exists almost everywhere and is horizontal. 

\begin{proposition}\label{limitishorizontal}
Suppose $\gamma\colon [0,1]\to\mathbb{F}$ is the curve from Proposition~\ref{limitislipschitz}.
For almost every $x\in [0,1]$, we have $\gamma'(x)\in \{\pm \lambda_0 X_{1}(\gamma(x)), \pm \lambda_0 X_{2}(\gamma(x)) \}$.
\end{proposition}

\begin{proof}
Recall that, for every $n\in \mathbb{N}$, the collection $\{J_j^n\}_{j=1}^{N_n}$ partitions $[0,1]$ into $N_n$ consecutive intervals of equal length. Suppose that for each $n$ and $j$ the collection $\{I_{j,i}^n\}_{i=1}^{3\cdot 5^{n}+2}$ is a partition of $J_j^n$ into $3\cdot 5^n+2$ consecutive intervals of equal length.
    Note that $m(I_{j,i}^n)=(N_n(3\cdot 5^{n}+2))^{-1}$ for each $i$, $j$, and $n$.

    Fix $n \in \mathbb{N}$, and set $F_{n} = \bigcup_{j =1}^{N_n} (I_{j,(5^{n}+1)}^n \cup I_{j,(2\cdot 5^n+2)}^n)$. According to Proposition~\ref{p-inductiveconstruction}(\ref{item-segments2}), $F_n$ consists of those subintervals of $[0,1]$ on which 
    $\gamma_{n+1}$ is not simply a horizontal $\lambda_{n+1}$-segment.
    Notice 
    $$
    m(F_n) \leq \sum_{j =1}^{N_n} m(I_{j,(5^n+1)}^n \cup I_{j,(2\cdot 5^n+2)}^n) = \frac{2N_n}{N_n(3\cdot 5^{n} + 2)} \leq 5^{-n}.
    $$
    Therefore $\sum_{n=1}^\infty m(F_{n})<\infty$, so, by the Borel-Cantelli lemma, $m(\bigcap_{n=1}^\infty \bigcup_{k\geq n}F_{k})=0$. In other words,
    if we write $E_{n}=\bigcap_{k\geq n}\left([0,1] \setminus F_k \right)$ then $m(\bigcup_{n=1}^{\infty} E_{n})=1$.
    Let $E$ denote the set of all points $x\in [0,1]$ such that
    \begin{itemize}
        \item $x \in \bigcup_{n=1}^\infty E_{n}$,
        \item $\gamma_n'(x)$ exists and satisfies the horizontality condition \eqref{lifteq} for all $n \in \mathbb{N}$,
        \item $x$ is not an endpoint of $J_{j}^{n}$ or $I_{j,i}^{n}$ for all $n \in \mathbb{N}$ and any
        $j \in \{1,\dots,N_n\}$
        and 
        $i \in \{1,\dots,3 \cdot 5^n + 2\}$.
    \end{itemize}
    Note that $m([0,1] \setminus E)=0$,
    so it suffices to verify that 
    $\gamma'(x)\in \{\pm \lambda_0 X_{1}(\gamma(x)), \pm \lambda_0 X_{2}(\gamma(x)) \}$ for all $x \in E$.
    Fix such an $x$ and choose $n$ such that $x \in E_n$.
    By definition, 
    $x \notin F_k$ for all $k \geq n$, and hence
    $x \notin I_{j,(5^k+1)}^k$ and  $x \notin I_{j,(2 \cdot 5^k+2)}^k$ for any $k \geq n$ and for all $j \in \{1,\dots,N_k\}$. 

\medskip

    \noindent \emph{Claim:} There exists some $V\in \{\pm X_{1}, \pm X_{2}\}$ such that, for every $k \geq n$, there is 
    some $j \in \{1,\dots,N_k\}$ such that $x$ lies in the interior of $J_j^k$ and
    $\gamma_{k}|_{J_j^k}$ is a 
    horizontal $\lambda_{k}V$-segment. 

    \begin{proof}[Proof of Claim]
    Choose $j \in \{1,\dots,N_n\}$ such that $x$ is in the interior of $J_{j}^{n}$.
    According to \eqref{inductiveproperty},
    there is some $V\in \{\pm X_{1}, \pm X_{2}\}$ such that
    $\gamma_{n}|_{J_{j}^{n}}$ is a horizontal $\lambda_{n}V$-segment. We prove the claim by induction.

    Suppose $k \geq n$ and $j \in \{1,\dots,N_k\}$ such that $x$ is in the interior of $J_{j}^{k}$ and 
    $\gamma_{k}|_{J_{j}^{k}}$ is a horizontal $\lambda_{k}V$-segment (for the same $V$ as above).
    Since $x \notin F_{k}$, there is some $i \in \{1,\dots,3 \cdot 5^{k}+2\}$ with $i \neq 5^{k}+1$ and $i \neq 2 \cdot 5^{k}+2$ such that
    $x$ is in the interior of $I_{j,i}^{{k}}\subset J_{j}^{{k}}$,
    and Proposition~\ref{p-inductiveconstruction}(\ref{item-segments2}) implies that $\gamma_{{k}+1}|_{I_{j,i}^{{k}}}$ is a horizontal $\lambda_{{k}+1}V$-segment. Since $3\cdot 5^{{k}}+2$ divides $N_{{k}+1}/N_{{k}}$, the interval $I_{j,i}^{{k}}$ can be written as a union of 
    consecutive intervals of the form $J_{\ell}^{{k}+1}$ for $\tfrac{N_{{k}+1}}{N_{{k}}(3 \cdot 5^{k} + 2)}$ consecutive indices $\ell \in \{1,\dots,N_{k+1} \}$. Therefore, $x$ is contained in the interior of $J_{\ell_0}^{{k}+1}$ for some $\ell_0 \in \{1,\dots,N_{{k}+1} \}$, and $\gamma_{{k}+1}|_{J_{\ell_0}^{{k}+1}}$ is a horizontal $\lambda_{{k}+1}V$-segment. The claim then follows inductively. 
    \end{proof}

According to the claim and \eqref{e-nicederivativeform},
we conclude that there is a fixed $V \in \{\pm X_1,\pm X_2\}$ independent of $k$
such that $\gamma_k'(x) = \lambda_kV(\gamma_k(x))$ for every $k \geq n$.
Thus, we can define a map $g\colon [0,1] \to \mathbb{R}^5$ such that,
for all $x \in E$,
\[
g(x)=\lim_{k\to\infty} \gamma_k'(x) = 
\lim_{k\to\infty} \lambda_k V(\gamma_k(x))
=
\lambda_0V(\gamma(x))
\]
and $g|_{[0,1] \setminus E} \equiv 0$.
Since $\gamma_{k}$ is $\lambda_{k}$-Lipschitz, for all $k \in \mathbb{N}$, 
we have that 
$\text{diam}(\gamma_k([0,1])) \leq \lambda_k < 2$.
Thus \eqref{distancecomparison} implies that 
there is a uniform bound on $|\gamma_k'|_E|$.
It then follows from the dominated convergence theorem that, for all $x,y \in E$ with $x \leq y$,
\begin{align*}
    \gamma(y) - \gamma(x) = \lim_{n \to \infty} (\gamma_n(y)-\gamma_n(x)) = \lim_{n \to \infty}\int_{[x,y] \cap E} \gamma_{n}'
    =
    \int_{[x,y] \cap E} g
\end{align*}
since $E$ is a set of full measure in $[0,1]$.
Therefore, for almost every $x \in E$, 
and hence almost every $x \in [0,1]$,
$\gamma'(x) =g(x)=\lambda_0 V(\gamma(x))$ for some $V \in \{\pm X_1,\pm X_2\}$. 
\end{proof}

With our curve $\gamma$ in hand, we are finally ready to prove the main result of the paper.

\begin{theorem*}[Restatement of Theorem \ref{lusinzero}]
There is a Lipschitz curve $\gamma\colon [0,1]\to \mathbb{F}$ such that, for every $C^{1}_H$ curve $\Gamma\colon [0,1]\to \mathbb{F}$, we have
\[
m\{t\in [0,1]:\Gamma(t)=\gamma(t)\}=0.
\]
\end{theorem*}

\begin{proof}
Let $\gamma\colon [0,1]\to\mathbb{F}$ be the curve from Proposition~\ref{limitislipschitz}, and fix a $C^{1}_H$ curve $\Gamma\colon [0,1]\to \mathbb{F}$.
We argue by contradiction. Let $S:=\{t\in (0,1): \Gamma(t)=\gamma(t)\}$ and suppose $m(S)>0$. According to Proposition~\ref{limitishorizontal}, we can choose a density point $t_{0}$ of $S$ such that $\gamma'(t_{0})$ exists and equals $\lambda_0 V(\gamma(t_0))$ for some $V\in \{\pm X_{1}, \pm X_{2}\}$. 
Note that $\Gamma'(t_{0})=\gamma'(t_{0})$. 
Indeed, since $t_0$ is a density point of $S$, we can choose a sequence $t_n \to t_0$ inside $S$ so that
\begin{align*}
    (t_n-t_0)^{-1}&\left| \Gamma(t_n) - \Gamma(t_0) - (t_n-t_0)\gamma'(t_0)\right|\\
    &=
    (t_n-t_0)^{-1}\left| \gamma(t_n) - \gamma(t_0) - (t_n-t_0)\gamma'(t_0)\right| \to 0 \text{ as } n \to \infty.
\end{align*}
We also have that either $|\Gamma_1'(t_0)| \geq \lambda_0$ or $|\Gamma_2'(t_0)| \geq \lambda_0$.
Indeed, this follows from the fact that $\Gamma'(t_0) = \gamma'(t_0) = \lambda_0 V(\Gamma(t_0))$ since one of the first two coordinates of $V(\Gamma(t_0))$ is equal to $\pm 1$ regardless of the choice of $V$.

Using the facts that $\Gamma'$ is continuous, $\lambda_0 > 1$, and $t_{0}$ is a density point of $S$, we may choose $\delta>0$ sufficiently small enough such that
\begin{enumerate}
    \item $[t_{0}-\delta,t_{0}+\delta]\subset [0,1]$,
    \item either $|\Gamma_1'| \geq \tfrac12$ on $[t_{0}-\delta,t_{0}+\delta]$ or $|\Gamma_2'| \geq \tfrac12$ on $[t_{0}-\delta,t_{0}+\delta]$,
    \item $m( [t_{0}-\delta,t_{0}+\delta] \setminus S  )<\delta/12$.
\end{enumerate}
As before, for each $n \in \mathbb{N}$, let $\{J_j^n\}_{j=1}^{N_n}$ be the partition of $[0,1]$ into $N_n$ consecutive intervals of equal length.

\medskip

\noindent {\bf Claim:} Fix $n$ large enough that $\delta N_{n}>3$. 
Then there exists some $j_0 \in \{1,\dots,N_n\}$ such that $J_{j_0}^n \subset [t_{0}-\delta,t_{0}+\delta]$ and $m( J_{j_0}^n \setminus S  )<(12N_{n})^{-1}$.

\begin{proof}[Proof of Claim]
 The interval $[t_{0}-\delta,t_{0}+\delta]$ has length $2\delta$ and the intervals $J_{j}^n$ have length $N_{n}^{-1}$. Hence $[t_{0}-\delta,t_{0}+\delta]$ contains at least $\lfloor 2\delta N_{n} \rfloor -2$ full intervals of the form $J_j^n$ for $j \in \{1,\dots,N_n\}$. Denote the set of such indices $j$ by $\mathcal{J}$, so that $|\mathcal{J}|\geq \lfloor 2\delta N_{n} \rfloor -2$. We then estimate as follows:
 \begin{align*}
 \frac{\delta}{12} > m( [t_{0}-\delta,t_{0}+\delta] \setminus S  )\geq m \left( \bigcup_{j\in \mathcal{J}} (J_j^n \setminus S) \right) =\sum_{j\in \mathcal{J}} m( J_j^n \setminus S ) \geq |\mathcal{J}|\min_{j\in \mathcal{J}}m(J_j^n\setminus S ).
 \end{align*}
 Hence, since the assumption $\delta N_{n}>3$ implies
 $2\delta N_n - 3 > \delta N_n$, we have
 \begin{align*}
\min_{j\in \mathcal{J}}m(J_j^n \setminus S ) 
&< \frac{\delta}{12|\mathcal{J}|} 
\leq 
\frac{\delta}{12 \left(2\delta N_{n}-3\right)}
<
\frac{1}{12 N_n}.
 \end{align*}
 Choosing an index $j_0 \in \mathcal{J}$ that realizes this minimum proves the claim.
\end{proof}

We use the claim to continue the proof of the proposition; fix $n$ and $j_0$ so that
\begin{equation}\label{bigger}
m\{t\in J_{j_0}^n:\Gamma(t)\neq \gamma(t)\} = m( J_{j_0}^n \setminus S  )< (12N_{n})^{-1}.
\end{equation} 
We will derive a contradiction by considering the behaviors of the curves $\gamma_{n+1}$ and $\gamma$ on $[a,b]:=J_{j_0}^n$. 
Recall from 
\eqref{e-nextstepwitheta}
that $\gamma_{n+1}|_{J_{j_0}^n} = \gamma_{n}(a)*\rho$, where $\rho$ is one of the curves $\alpha^+$, $\alpha^-$, $\beta^+$, or $\beta^-$ from Construction~\ref{c-construction}.

We will now apply Proposition~\ref{badintersect} to the curves $\eta = \gamma_n(a)^{-1} * \gamma$ and $\rho = \gamma_n(a)^{-1} * \gamma_{n+1}$ on $J_{j_0}^n = [a,b]$.
This can be done since \eqref{e-unifconv} and \eqref{e-sillyNnbound} give
\begin{align*}    
d_{c}(\eta(t),\rho(t))
=  
d_{c}(\gamma(t),\gamma_{n+1}(t))
\leq 
2\sum_{k=n+1}^\infty N_k^{-1}5^{-k}
&<\frac{2}{N_{n+1}} \sum_{k=1}^\infty 5^{-k}\\
&<
\frac{1}{10\kappa}\left( \frac{\lambda_{n} (b-a)}{24 \cdot 5^n }\right)^{3}
\end{align*}
for all $t \in [a,b]$ since $b-a=N_n^{-1}$.
Since $[a,b] \subset [t_0-\delta,t_0+\delta]$, we have either $|\Gamma_1'| \geq \tfrac12$ on $[a,b]$ or $|\Gamma_2'| \geq \tfrac12$ on $[a,b]$. 
Hence we may use Proposition~\ref{badintersect} with $\lambda = \lambda_n$, $Q = 5^n$, and $[a,b] = J_{j_0}^n$ applied to the $C^1_H$ curve
$t \mapsto \gamma_{n}(a)^{-1} * \Gamma(t)$
to conclude from \eqref{bigger} that
\begin{align*}
\frac{11}{12N_n} < 
m\{t\in J_{j_0}^n:\Gamma(t)=\gamma(t)\}
=
m\{t\in [a,b]:\gamma_{n}(a)^{-1} * \Gamma(t)=\eta(t)\} \leq \frac{4}{5N_{n}}
\end{align*}
which is a contradiction. Hence it must have been the case that $m(S)=0$, and this completes the proof.
\end{proof}

\section{A Purely $C_{H}^{1}$ 1-Unrectifiable Lipschitz Curve}\label{s-equivalence}

We now show how Theorem \ref{lusinzero} leads to Theorem~\ref{t-supermain}. In fact, we show that such an implication holds in any Carnot group, not only in $\mathbb{F}$. Throughout this section we fix an arbitrary Carnot group $\mathbb{G}$ written in either first or second exponential coordinates with respect to some fixed basis of $V_{1}$ as $\mathbb{R}^n$ with the Carnot-Carath\'{e}odory metric $d_c$. Denote by $\mathfrak{g}$ its Lie algebra and by $V_1$ the first (horizontal) layer of $\mathfrak{g}$.
We invite the reader who is not familiar with general Carnot group theory to substitute $\mathbb{F}$ for $\mathbb{G}$ with the coordinates used earlier wherever necessary.

We will prove the following result.

\begin{theorem}
    \label{t-movetounrect}
    Suppose there is a Lipschitz curve $\gamma \colon[0,1] \to \mathbb{G}$ such that 
    \begin{equation}
    \label{e-supergoal2}
m\{t\in [0,1]:\Gamma(t)=\gamma(t)\}=0
\end{equation}
    for any $C^1_H$ curve $\Gamma \colon [0,1] \to \mathbb{G}$.
    Then $\gamma([0,1])$ is purely $C^1_H$ 1-unrectifiable.
\end{theorem}

Note that whether a Lipschitz curve $\gamma$ satisfying \eqref{e-supergoal2} exists depends on the Carnot group $\mathbb{G}$. For example, in the Heisenberg group or the Engel group, there is no such curve \cite{SZ25,Zim18}. However, in Theorem~\ref{lusinzero} above, we showed that such a curve does exist in $\mathbb{F}$.

Before proving Theorem \ref{t-movetounrect}, we establish several facts that will be useful in the proof.
The first is the classical Lusin approximation of a Lipschitz function on $\mathbb{R}$ by a $C^1$ function.
See \cite[6.6.1 Theorem 1]{EG15} or \cite[Theorem~3.1.16]{Fed69} for a reference.

\begin{lemma}
\label{l-Federer}
    Suppose $K \subset \mathbb{R}$ is measurable with $m(K)>0$ and $f\colon K \to \mathbb{R}$ is Lipschitz. 
    Then there is a compact set $A \subset K$ with $m(A)>0$ 
    and a $C^1$ function $\varphi\colon \mathbb{R} \to \mathbb{R}$ such that $\varphi(x)=f(x)$ for all $x \in A$.
\end{lemma}

\begin{proof}
Extend $f$ to a Lipschitz function $F\colon \mathbb{R} \to \mathbb{R}$.
By Rademacher's Theorem,
$F'$ exists almost everywhere.
Applying 
Lusin's (classical) Theorem to $F'$ on $K$
and also
applying Egorov's theorem to the difference quotients 
$$
\psi_n(x) := \sup_{y \in K \cap B(x,\frac{1}{n})} \left| \frac{F(y) - F(x) - (y-x)F'(x)}{y-x}\right|
$$
provides a compact subset $A \subset K$ with $m(A)>0$ on which $F'$ is continuous and such that 
$
\frac{F(y) - F(x)}{y-x}
$
converges uniformly to $F'$ on $A$.
Whitney's Extension Theorem \cite{Whi34} then guarantees that there exists a $C^1$ function $\varphi \colon \mathbb{R}\to \mathbb{R}$
such that, for all $x \in A$, we have
$\varphi(x) = F(x) = f(x)$.
\end{proof}


Suppose $I\subset \mathbb{R}$ is an interval and $\gamma\colon I \to \mathbb{G}$ is a Lipschitz curve.
For the rest of this section, we will write 
$|\gamma'|_H$ 
and $|\gamma'|_{\mathbb{R}^n}$ to denote the horizontal and Euclidean speeds
of $\gamma$ respectively when they exist. We point out that when $\mathbb{G}=\mathbb{F}$, we have $n=5$ and 
$|\gamma'|_H = \sqrt{(\gamma_1')^2+(\gamma_2')^2}$ a.e. as in Lemma~\ref{lifting}. We also recall that the Hausdorff measure $\mathcal{H}^1$ is defined with respect to the CC-metric $d_c$.

\begin{remark}\label{extensiontoR}
Suppose $\Gamma\colon [0,1]\to \mathbb{G}$ is a $C_{H}^{1}$ curve. By extension using horizontal lines, we can extend $\Gamma$ to a curve defined on $\mathbb{R}$ that is both $C_{H}^{1}$ and Lipschitz.
More precisely, let $Z_0,Z_1 \in V_1 \subset \mathfrak{g}$ be the vectors satisfying $\Gamma'(0)=Z_0(\Gamma(0))$ and $\Gamma'(1)=Z_1(\Gamma(1))$. We extend $\Gamma$ using the formula
\[
\Gamma(t) = 
\begin{cases}
    \Gamma(0)*\exp(tZ_0) & \text{if } t\leq0\\
    \Gamma(1)*\exp((t-1)Z_1) & \text{if } t \geq 1.
\end{cases}
\]
Using the definition of the exponential map, 
it follows that $\Gamma$ is $C_{H}^{1}$ on $\mathbb{R}$. By Lemma \ref{l-C1HisLip} it then follows that $\Gamma$ is Lipschitz on $\mathbb{R}$.
\end{remark}

The final tool we need will be the area formula of Kirchheim for maps from Euclidean sets into arbitrary metric spaces. We will not need its full strength, so we provide the simplified statement that is relevant in our context in the proof below.

\begin{lemma}\label{applyareaformula}
Let $\eta\colon \mathbb{R}\to \mathbb{G}$ be a Lipschitz curve in $\mathbb{G}$ and $E\subset \eta(\mathbb{R})$ be a Borel set such that $\mathcal{H}^{1}(E)>0$. 
Then the set $T:=\{t \in \eta^{-1}(E) : \eta'(t) \mbox{ exists and }\eta'(t) \neq 0\}$ is Borel measurable and has strictly positive Lebesgue measure.
\end{lemma}

\begin{proof} First notice that $T$ is Borel because $\eta^{-1}(E)$ is the continuous preimage of a Borel set and $\eta'$ is Borel measurable.
Recall the {\em metric derivative} $md(\eta)$ defined by
\[
md(\eta)(t) := \lim_{h \to 0} \frac{d_c(\eta(t+h),\eta(t))}{h},
\] 
which exists for almost every $t \in \mathbb{R}$ as $\eta$ is Lipschitz \cite[Theorem 4.1.6]{AT04}.
We apply the area formula of Kirchheim \cite[Theorem 7]{K94}. In our situation, this implies the following: for every Borel set $E\subset \eta(\mathbb{R})$,
\begin{equation}
\label{e-areaform}
\int_{\eta^{-1}(E)} md(\eta) (t) dt
=\int_{\mathbb{G}} \#(\eta^{-1}(p) \cap \eta^{-1}(E)) \, d \mathcal{H}^1(p)
\geq \mathcal{H}^1(E).
\end{equation}
To see the right hand side inequality, observe that, if $p\in E\subset \eta(\mathbb{R})$, then $p=\eta(t)$ for some $t\in \eta^{-1}(p)\cap \eta^{-1}(E)$. This implies $\#(\eta^{-1}(p) \cap \eta^{-1}(E))\geq 1$.

For almost every $t \in I$, the definition of the metric $d_c$ and the 
Lebesgue differentiation theorem 
applied to the function $|\eta'|_{H} \in L^1_{loc}(\mathbb{R})$
give
\begin{equation}
\label{e-metricderiv}
\begin{aligned}
md(\eta)(t) = \lim_{h\searrow 0} \frac{d_c(\eta(t+h),\eta(t))}{h}
\leq
\lim_{h\searrow 0}\frac{1}{h} \ell(\eta|_{[t,t+h]})
&=
\lim_{h\searrow 0}\frac{1}{h}\int_t^{t+h} |\eta'|_H\\
&= |\eta'(t)|_H
\leq |\eta'(t)|_{\mathbb{R}^n}.
\end{aligned}
\end{equation}
To conclude, suppose for a contradiction that $m(T)=0$. Then $\eta'(t)=0$ for almost every $t\in \eta^{-1}(E)$. Hence by \eqref{e-areaform} and \eqref{e-metricderiv}, we must have $\mathcal{H}^{1}(E)=0$, which violates our hypothesis.
\end{proof}

Before proving Theorem \ref{t-movetounrect} we briefly describe the proof.
Given $\gamma$ and $\Gamma$ as in the statement, let $E$ be the set of points in the target where these curves overlap. Our goal is to show that $\mathcal{H}^1(E)=0$. 
If this was not the case, then Lemma~\ref{applyareaformula} would imply that $\Gamma'$ is bounded away from 0 on an appropriate interval, and hence $\Gamma$ would be bi-Lipschitz into $E$ on a set of positive measure $Z$.
If $\gamma$ coincided with $\Gamma$ on $Z$, this would be a contradiction. However, there is no guarantee that these curves meet within the same time-frame i.e. that $\gamma(Z) \subset E$. To overcome this, we apply Lemma~\ref{l-Federer} to the Lipschitz map $\Gamma^{-1} \circ \gamma$ on a positive measure set $A \subset \mathbb{R}$ to smoothly reparameterize $\Gamma$ and force $\gamma$ to coincide with this new curve on $A$. This contradicts our assumption on $\gamma$ and completes the proof.

\begin{proof}[Proof of Theorem~\ref{t-movetounrect}]
Suppose
$\gamma\colon [0,1] \to \mathbb{G}$ is a Lipschitz curve such that \eqref{e-supergoal2} holds for all $C^1_H$ curves $\Gamma\colon [0,1] \to \mathbb{G}$.
Fix a $C^1_H$ curve $\Gamma\colon [0,1] \to \mathbb{G}$. Using Remark \ref{extensiontoR}, we can extend $\Gamma$ to a curve defined on $\mathbb{R}$ that is both $C^{1}_{H}$ and Lipschitz. Let $E = \gamma([0,1]) \cap \Gamma(\mathbb{R})$.
It suffices to show that $\mathcal{H}^1(E)=0$.

Suppose instead that $\mathcal{H}^1(E)>0$. Clearly $E\subset \Gamma(\mathbb{R})$ and $E$ is Borel because $\gamma$ and $\Gamma$ are continuous. By Lemma~\ref{applyareaformula}, the set
\[
T:=\{t \in \Gamma^{-1}(E) : \Gamma'(t) \mbox{ exists and }\Gamma'(t) \neq 0\}
\]
is Borel measurable and has strictly positive Lebesgue measure.
Hence we can choose $t_{0}\in T$ that is a density point of $T$. 

Next, we choose an appropriate interval $I$ on which $\Gamma|_I$ is bi-Lipschitz.
To do this, use the fact that $\Gamma'$ is continuous and $\Gamma'(t_0)\neq 0$ to choose a compact interval $I$ of positive length containing $t_{0}$ 
such that $|\Gamma_j'| \geq c_1$ on $I$ for some $c_1>0$ and some $j \in \{1,\dots,n\}$.
Since $\Gamma(I)$ is compact, there is a constant $c_2>0$ such that
$
d_c(\Gamma(t),\Gamma(s))\geq c_2|\Gamma(t)-\Gamma(s)|
$
for all $s,t \in I$. Note that when $\mathbb{G} = \mathbb{F}$ this follows from \eqref{disugdist}. For a proof in general Carnot groups, see the references after \eqref{disugdist}. 
Thus for any points $a<b$ in $I$, we have
$$
c_2^{-1} d_c(\Gamma(b),\Gamma(a)) \geq |\Gamma(b)-\Gamma(a)| \geq |\Gamma_j(b)-\Gamma_j(a)|= \left| \int_a^b \Gamma_j'\right| = \int_a^b |\Gamma_j'|
\geq c_1(b-a).
$$

Since $t_0$ is a density point of $T$, the intersection $Z := I \cap T \subset \Gamma^{-1}(E)$ has positive Lebesgue measure, so
$\mathcal{H}^1(\Gamma(Z))>0$ as $\Gamma|_{I}$ is bi-Lipschitz. 
Set $K:= \gamma^{-1}(\Gamma(Z))$. This set is Borel. Indeed, $\Gamma|_Z$ is bi-Lipschitz, so 
$\Gamma^{-1}|_{\Gamma(Z)}$ is continuous. Thus, since $Z$ is Borel, $\Gamma(Z) = (\Gamma^{-1}|_{\Gamma(Z)})^{-1}(Z)$ is Borel, and hence $K$ is Borel since $\gamma$ is also continuous. Further, since $\Gamma(Z)\subset \gamma([0,1])$ we have $\gamma(K)=\Gamma(Z)$.
Therefore, since $\gamma$ is Lipschitz,
it follows that $m(K)>0$.


Notice $\Gamma^{-1}\circ \gamma$ is well defined and Lipschitz on $K$.
Hence by Lemma~\ref{l-Federer},
there is a compact subset $A \subset K$ with $m(A)>0$
and a $C^1$ function $\varphi\colon \mathbb{R} \to \mathbb{R}$
such that $\varphi(x) = \Gamma^{-1}(\gamma(x))$ for all $x \in A$.
Now define $\tilde{\Gamma}\colon  \mathbb{R} \to \mathbb{G}$ by $\tilde{\Gamma} = \Gamma \circ \varphi$, noting that $\tilde{\Gamma}$ is well defined since $\Gamma$ was extended to have domain $\mathbb{R}$.
Observe that $\tilde{\Gamma}$ is $C^1_H$. To see this, write $\Gamma'(t) = X_t(\Gamma(t))$ for each $t \in \mathbb{R}$, where $X_t \in V_1 \subset \mathfrak{g}$. 
Note that $X_t(\Gamma(t))$ is continuous in $t$ since $\Gamma'$ is continuous.
Then for any  $t \in \mathbb{R}$, 
\[
\tilde{\Gamma}'(t) = \varphi'(t)\Gamma'(\varphi(t))
 = \varphi'(t) X_{\varphi(t)}(\Gamma(\varphi(t))) \in V_{1},
\]
and $\tilde{\Gamma}'$ is continuous since $\varphi \in C^1(\mathbb{R})$.
Moreover for all $t\in A$, since $\varphi|_{A} = (\Gamma^{-1} \circ \gamma)|_A$,
\[
\tilde{\Gamma}(t) = \Gamma(\varphi(t)) = \Gamma(\Gamma^{-1}(\gamma(t)))= \gamma(t).
\]
However, applying \eqref{e-supergoal2} with the $C_{H}^{1}$ curve $\tilde{\Gamma}$ gives
\begin{align*}
    0 = m \{ t \in [0,1] : \tilde{\Gamma}(t) = \gamma(t)\}
    &\geq
    m (A) > 0,
\end{align*}
which is a contradiction.
Therefore, $\mathcal{H}^1(E)=0$ and the proof is complete.
\end{proof}

\begin{proof}[Proof of Theorem~\ref{t-supermain}]
    Combining Theorem~\ref{lusinzero} and Theorem~\ref{t-movetounrect} suffices.
\end{proof}

\printbibliography

\end{document}